\documentclass[12pt,a4paper,reqno]{amsart}
\usepackage{amsmath}
\usepackage{amsthm}
\usepackage{amsfonts}
\usepackage{amssymb}
\usepackage{color}
\usepackage{mathrsfs}
\usepackage{graphicx}

\numberwithin{equation}{section}

\theoremstyle{plain}
\newtheorem{thm}{Theorem}

\theoremstyle{plain}
\newtheorem{lem}{Lemma}[section]

\theoremstyle{plain}
\newtheorem*{untheorem}{Theorem}

\theoremstyle{plain}

\theoremstyle{remark}
\newtheorem*{remark}{Remark}

\theoremstyle{definition}
\newtheorem*{case0}{Simple Case}

\theoremstyle{definition}
\newtheorem*{case1a}{Case 1A}

\theoremstyle{definition}
\newtheorem*{case1b}{Case 1B}

\theoremstyle{definition}
\newtheorem*{case2a}{Case 2A}

\theoremstyle{definition}
\newtheorem*{case2b}{Case 2B}

\theoremstyle{definition}
\newtheorem*{case3a}{Case 3A}

\theoremstyle{definition}
\newtheorem*{case3b}{Case 3B}

\theoremstyle{definition}
\newtheorem*{casetau0a}{Case $\bftau$A}

\theoremstyle{definition}
\newtheorem*{casetau0b}{Case $\bftau$B}

\theoremstyle{definition}
\newtheorem*{casetau1a}{Case ($\bftau+\bone$)A}

\theoremstyle{definition}
\newtheorem*{casetau1b}{Case ($\bftau+\bone$)B}

\theoremstyle{definition}
\newtheorem*{cases1b}{Case ($\bfits-\bone$)B}

\def\bfits{\boldsymbol{s}}

\def\bftau{{\boldsymbol\tau}}

\def\bone{\mathbf{1}}

\def\AAA{\mathcal{A}}

\def\III{\mathcal{I}}

\def\LLL{\mathcal{L}}

\def\PPP{\mathcal{P}}

\DeclareMathOperator{\length}{length}

\renewcommand{\le}{\leqslant}
\renewcommand{\ge}{\geqslant}

\title[Generalization of a density theorem of Khinchin]
{Generalization of a density theorem of\\
Khinchin and diophantine approximation}

\author[Beck]{J. Beck}

\address{Department of Mathematics, Rutgers University, Hill Center for the Mathematical Sciences, Piscataway NJ 08854, USA}

\email{jbeck@math.rutgers.edu}

\author[Chen]{W.W.L. Chen}

\address{Department of Mathematics and Statistics, Macquarie University, Sydney NSW 2109, Australia}

\email{william.chen@mq.edu.au}



\begin{document}

\begin{abstract}
The continuous version of a fundamental result of Khinchin says that a half-infinite torus line in the unit square $[0,1]^2$ exhibits superdensity, which is a best form of time-quantitative density, if and only if the slope of the geodesic is a badly approximable number.
In this paper, we give a proof of the extension of this result of Khinchin to the case when the unit torus $[0,1]^2$ is replaced by a finite polysquare surface, or square tiled surface.
The argument is based on diophantine approximation and continued fractions, traditional tools in number theory.
In particular, we use the famous $3$-distance theorem in diophantine approximation combined with an iterative process.
In short, this is a very number-theoretic study of a very number-theoretic problem.

This paper improves on an earlier result of the authors and Yang~\cite{BCY1} where it is shown that badly approximable numbers that satisfy a quite severe technical restriction on the digits of their continued fractions lead to superdense geodesics.
Here we overcome this technical impediment.

This paper is self-contained, and the reader does not need any knowledge of dynamical systems.
\end{abstract}

\maketitle

\thispagestyle{empty}

%
%

\section{Introduction}\label{sec1}

It is well known that the distribution of the irrational rotation sequence $n\alpha\bmod{1}$, $n=1,2,3,\ldots,$ is intimately related to the distribution of
half-infinite torus lines of slope $\alpha$ in the unit square $[0,1]^2$, \textit{i.e.}, geodesics of slope $\alpha$ on the unit torus $[0,1]^2$.
An old result of Khinchin \cite[Theorem~26]{K2} implies the following result concerning superdensity of geodesics; an alternative proof can be found in~\cite[Lemma~6.1.1]{BCY1}.

\begin{untheorem}[Khinchin]
Any half-infinite geodesic is superdense on the unit torus $[0,1]^2$ if and only if the slope of the geodesic is a badly approximable number.
\end{untheorem}

Superdensity is a time-quantitative criterion.
A half-infinite geodesic $\LLL(t)$, $t\ge0$, equipped with the usual arc-length parametrization, is superdense on the unit torus $[0,1]^2$ if there exists an absolute constant $C_1=C_1(\LLL)>0$ such that, for every integer $n\ge1$, the initial segment $\LLL(t)$, $0\le t\le C_1n$, of the geodesic gets
$(1/n)$-close to every point of $[0,1]^2$.

This concept of superdensity is a best possible form of time-quantitative density, in the sense that the linear length $C_1n$ cannot be replaced by any sublinear length $o(n)$ as $n\to\infty$.
For a simple proof of this; see \cite[Section~6.1]{BCY1}.

A very natural number-theoretic question concerns possible extension of the result of Khinchin by replacing the unit torus $[0,1]^2$ by a finite surface of a certain kind.

A finite \textit{polysquare region}, or a finite \textit{square tiled region} in the terminology of dynamical systems, is an arbitrary connected, but not necessarily simply-connected, polygon $P$ on the plane which is tiled with unit squares, assumed to be closed, that we call the \textit{atomic squares} of~$P$, and which satisfies the following conditions:

(i) Any two atomic squares in $P$ either are disjoint, or intersect at a single point, or have a common edge.

(ii) Any two atomic squares in $P$ are joined by a chain of atomic squares where any two neighbors in the chain have a common edge.

Note that $P$ may have \textit{holes}, and we also allow \textit{whole barriers} which are horizontal or vertical \textit{walls} that consist of one or more boundary edges of atomic squares.
For example, the polysquare region in the picture on the left in Figure~1.1 has $5$ atomic squares, whereas the polysquare region in the picture on the right in Figure~1.1 has $32$ atomic squares, $2$ holes as well as $3$ horizontal walls and $4$ vertical walls.

\begin{displaymath}
\begin{array}{c}
\includegraphics{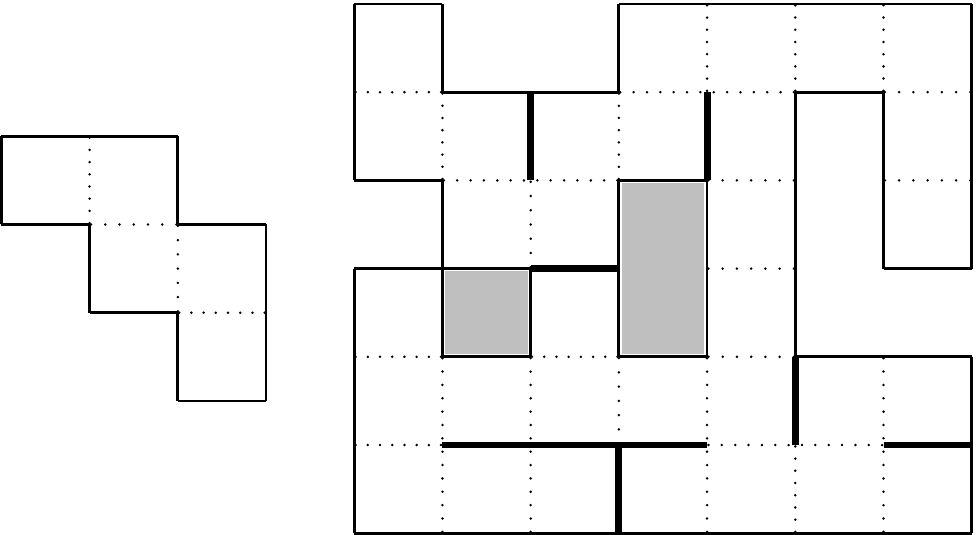}
\\
\mbox{Figure 1.1: some finite polysquare regions}
\end{array}
\end{displaymath}

Furthermore, a finite polysquare region can be converted to a finite \textit{polysquare surface}, or \textit{square tiled surface}, by pairwise identification of the boundary horizontal edges and pairwise identification of the boundary vertical edges.
Geodesic flow on this surface is thus $1$-direction linear flow.

The concept of superdensity can be extended to finite polysquare surfaces in a natural way.
A half-infinite geodesic $\LLL(t)$, $t\ge0$, equipped with the usual arc-length parametrization, is superdense on a finite polysquare surface $\PPP$ if there exists an absolute constant $C_1=C_1(\PPP;\LLL)>0$ such that, for every integer $n\ge1$, the initial segment $\LLL(t)$, $0\le t\le C_1n$, of the geodesic gets $(1/n)$-close to every point of~$\PPP$.

Using traditional tools in number theory based on diophantine approximation and continued fractions, we give a proof of the following extension of the result of Khinchin.

\begin{thm}\label{thm1}
Let $\PPP$ be an arbitrary finite polysquare surface.
A half-infinite geodesic that does not hit a vertex of $\PPP$ is superdense on $\PPP$ if and only if the slope of the geodesic is a badly approximable number.
\end{thm}

Theorem~\ref{thm1} is an \textit{if and only if} type result, where one of the two implications is a straightforward corollary of Khinchin's theorem.
Indeed, $1$-direction geodesic flow on a finite polysquare surface \textit{modulo one} becomes $1$-direction geodesic flow on the unit torus $[0,1]^2$, which implies that a superdense geodesic must have a badly approximable slope.
The hard task is to prove the converse, that every badly approximable slope leads to superdensity.

A finite polysquare surface may have singularities.
These then make the system \textit{non-integrable} and the analysis much harder.
A pioneering result in this direction concerns geodesics on a large class of surfaces, which we state below in the special case of finite polysquare surfaces.
This result is time-qualitative in nature, in that it does not give any indication on how long it takes for the geodesic to get within a given distance of a given point in~$\PPP$.
Indeed, the traditional approach from the viewpoint of dynamical systems is based on application of results such as Birkhoff's ergodic theorem which are essentially time-qualitative in nature.
Lacking an error term, they do not appear to lead naturally to time-quantitative statements.

\begin{untheorem}[Katok--Zemlyakov~\cite{KZ75}]
Apart from a countable set of directions, any half-infinite geodesic on a finite polysquare surface $\PPP$ is dense unless it hits a vertex of $\PPP$ and becomes undefined.
\end{untheorem}

Recall that an irrational number~$\alpha$, with continued fraction
\begin{equation}\label{eq1.1}
\alpha=[a_0;a_1,a_2,a_3,\ldots]=a_0+\frac{1}{a_1+\frac{1}{a_2+\frac{1}{a_3+\cdots}}},
\end{equation}
is said to be badly approximable if there exists a constant $A$ such that the continued fraction digits $a_i\le A$ for every $i=0,1,2,3,\ldots.$

In a recent series of papers \cite{BDY1,BDY2,BCY1,BCY2}, the authors and their co-authors are able to establish many results concerning the long-term \textit{time-quantitative} behavior in many flat systems.
In particular, a weaker form of Theorem~\ref{thm1} is established, where it is shown that superdensity follows if the slope $\alpha$ given by \eqref{eq1.1} satisfies the additional technical requirement that the digits $a_0,a_1,a_2,a_3,\ldots$ are all integer multiples of the street-LCM of the finite polysquare surface under consideration.
The street-LCM of a finite polysquare surface is the lowest common multiple of the lengths of the horizontal and the vertical streets of the surface.
While this excludes many badly approximable slopes, the method nevertheless gives an uncountable set of slopes which guarantee superdensity.
In Theorem~\ref{thm1}, we remove this technical impediment.

The proof of Theorem~\ref{thm1} here, however, is completely different from our earlier technique.
However, the two different approaches share a common characteristic, in that neither is based on the traditional application of ergodic theory in the earlier study of density and uniformity using traditional techniques in dynamical systems.
Instead, we appeal to a non-ergodic approach, and base our arguments on number theory and combinatorics.

We thus have two methods to prove superdensity.
They are not comparable, and have different advantages.
For instance, the shortline method developed in~\cite{BCY1} works beyond geodesics on polysquare surfaces, and can give superdensity for geodesics on any regular polygon surface.
We do not see how to we can establish such a result with the method of this paper.

Our proof here of Theorem~\ref{thm1} is elementary but not simple.
We therefore start by illustrating the ideas by studying the special case of the L-surface, which is arguably the simplest non-integrable polysquare surface.

The picture on the left in Figure~1.2 shows the L-shape region composed of $3$ unit squares.
Furthermore, it shows the L-graph, which is an undirected planar graph with $8$ vertices, $5$ horizontal edges and $5$ vertical edges.
Using this picture and identifying the edges, we can reduce the number of vertices.
On identifying the edges~$h_1$, we see that $A=G$ and $B=F$.
On identifying the edges~$h_2$, we see that $B=E$ and $C=D$.
On identifying the edges~$v_1$, we see that $A=C$ and $H=D$.
On identifying the edges~$v_2$, we see that $H=E$ and $G=F$.
Thus
\begin{displaymath}
A=G=F=B=E=H=D=C,
\end{displaymath}
so that all the vertices are identified with each other, and we have essentially only $1$ vertex.
This single vertex is a split singularity point of the geodesic flow, explaining why it is a \textit{non-integrable} system; see the two geodesics in the picture on the right in Figure~1.2.

Furthermore, the surface has $3$ faces and, after identification, $6$ edges, so the Euler characteristic is $\chi=1-6+3=-2$.
The genus $g$ is obtained from the well known formula $g=1-(\chi/2)=1+1=2$, and so the surface is homeomorphic to a $2$-holed torus using the classification theorem of closed surfaces.
We call this the \textit{L-surface}.
It is equipped with a flat metric, and the curvature is zero on every square face.
The two geodesics in the picture on the right in Figure~1.2 illustrate why the vertex $E$, and hence every other vertex, is a split singularity of the geodesic flow.
The L-surface is classified as a Riemann surface with a singular point.

\begin{displaymath}
\begin{array}{c}
\includegraphics{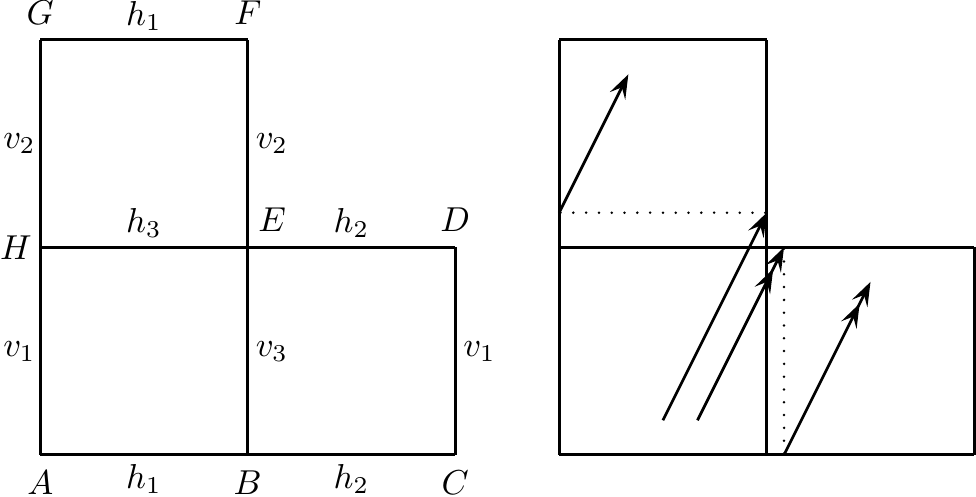}
\\
\mbox{Figure 1.2: the L-surface and two geodesics}
\end{array}
\end{displaymath}

As an analog of half-infinite geodesics on the unit torus $[0,1]^2$, we can consider half-infinite geodesics on the L-surface.
A particle moves on the geodesic with unit speed, so that time equals distance.
If a geodesic on the L-surface has irrational slope and never hits the singular point, then in the L-shape it is represented as a union of infinitely many \textit{parallel} line segments.
A geodesic is uniquely determined by one of its points and its constant velocity vector, just like a geodesic on the unit torus $[0,1]^2$.

We shall first prove the following special case of Theorem~\ref{thm1}.

\begin{thm}\label{thm2}
If a half-infinite geodesic that does not hit a vertex of the L-surface has a slope that is a badly approximable number, then it is superdense on the L-surface.
\end{thm}

We can summarize the proof of Theorem~\ref{thm2}, and hence also Theorem~\ref{thm1}, in a nutshell as a \textit{careful} use of the classical tool of continued fractions.

Before we begin our proof of Theorem~\ref{thm2}, however, we make some comments.

First of all, the situation is completely different if we consider uniform distribution instead of density.

Corresponding to superdensity, we can define superuniformity as a best form of time-quantitative uniformity, when the relevant discrepancy is of logarithmic size compared to the length of the geodesic.
For half-infinite geodesics on the unit torus $[0,1]^2$, the result of Khinchin says that a geodesic on the unit torus is superuniform if and only if its slope is a badly approximable number.

This, however, does not remain the case if we consider half-infinite geodesics on the L-surface.
While any half-infinite geodesic with a badly approximable slope is uniformly distributed on the L-surface, the rate of convergence to uniformity can be vastly different for two distinct badly approximable slopes.
For example, it is shown in \cite{BDY1,BDY2} that a half-infinite geodesic of slope $\alpha=\sqrt{2}$ on the L-surface is superuniform, whereas a half-infinite geodesic of slope $\alpha=1+\sqrt{2}$ on the L-surface exhibits discrepancy greater than random square-root size.

Secondly, Theorem~\ref{thm1} has analogs for billiard orbits in finite polysquare regions and for geodesics on surfaces of finite simply-connected polycube regions.
Billiard orbits in finite polysquare regions and geodesics on finite polysquare surfaces are related by the concept of unfolding due to  K\"{o}nig and
Sz\"{u}cs~\cite{KS}, first introduced to the unit square, leading to $4$-fold covering by reflection across a horizontal axis and across a vertical axis.
For an illustration, see also \cite[Section~1.3]{BDY1}.
As to surfaces of simply-connected polycube regions, the simplest example is the surface of the unit cube.
Geodesic flow on such a surface is $4$-direction geodesic flow, and it can be related to $1$-direction geodesic flow on a surface obtained by combining $4$ rotated copies of this surface in a suitable way.
For an illustration, see \cite[Example~7.2.4]{BCY2}.

Finally, it has been drawn to our attention that there is perhaps a possibility of establishing results such as Theorem~\ref{thm1} by using the ideas of Teichm\"{u}ller dynamics, a very powerful tool in dynamical systems.
However, such techniques are beyond the reach of many who are not experts in that area.
Here we have chosen to prove a number-theoretic result by using traditional number-theoretic methods.

%
%

\section{Some prerequisites}\label{sec2}

Without loss of generality, we assume that the slope of the geodesic is greater than~$1$.
Suppose that the geodesic has slope $1/\alpha$, where $0<\alpha<1$ is irrational.

Our first tool is an interval exchange transformation $T=T_{\alpha}$ which encodes the information concerning the particular order with which a geodesic of slope $1/\alpha$ keeps hitting the horizontal edges $h_1,h_2,h_3$.

The first step involves identifying the horizontal edges $h_1,h_2,h_3$ with unit intervals by making use of the correspondences
\begin{displaymath}
h_1=[0,1),
\quad
h_2=[1,2),
\quad
h_3=[2,3),
\end{displaymath}
perhaps somewhat abusing notation, as shown in Figure~2.1.

\begin{displaymath}
\begin{array}{c}
\includegraphics{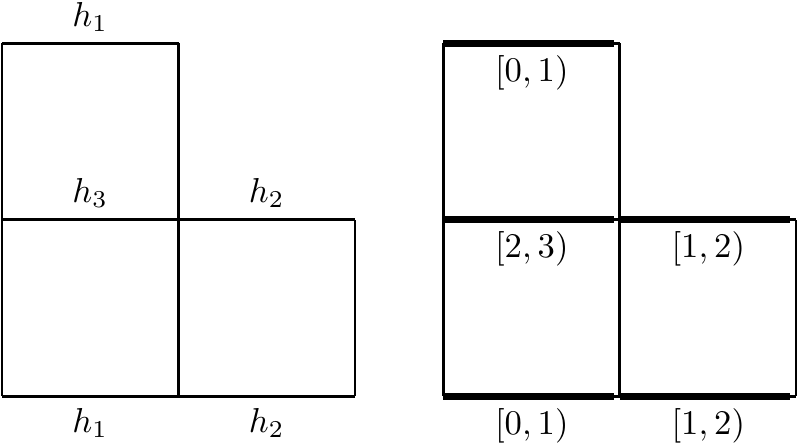}
\\
\mbox{Figure 2.1: representing horizontal edges of the L-surface by intervals}
\end{array}
\end{displaymath}

We now consider the piecewise linear map $T=T_{\alpha}$ defined according to Figure~2.2.

\begin{displaymath}
\begin{array}{c}
\includegraphics{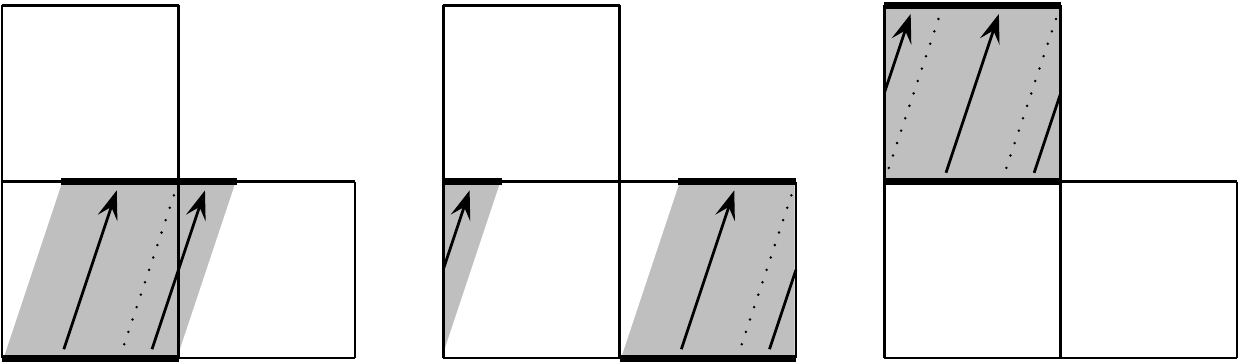}
\\
\mbox{Figure 2.2: the interval exchange transformation $T=T_{\alpha}$}
\end{array}
\end{displaymath}

This is called the interval exchange transformation.
More precisely, we have
\begin{align}
T([0,1-\alpha))=[2+\alpha,3),
&\quad
T([1-\alpha,1))=[1,1+\alpha),
\label{eq2.1}
\\
T([1,2-\alpha))=[1+\alpha,2),
&\quad
T([2-\alpha,2))=[2,2+\alpha),
\label{eq2.2}
\\
T([2,3-\alpha))=[\alpha,1),
&\quad
T([3-\alpha,3))=[0,\alpha),
\label{eq2.3}
\end{align}
where, for instance, $T([0,1-\alpha))=[2+\alpha,3)$ describes the $1/\alpha$ flow mapping the part $[0,1-\alpha)$ in $h_1=[0,1)$ to the part
$[2+\alpha,3)$ in $h_3=[2,3)$ linearly in the form
\begin{displaymath}
Tx=2+\alpha+x,\quad
x\in[0,1-\alpha),
\end{displaymath}
and similarly for the rest in \eqref{eq2.1}--\eqref{eq2.3}.

The novelty is that $T$ acts on the longer interval $[0,3)$ instead of the unit interval $[0,1)$, but if we consider $T$ modulo~$1$, then it acts simply as an $\alpha$-shift, or irrational rotation, in the unit interval.

We next consider our main idea, which involves continued fractions.
Consider an irrational number
\begin{equation}\label{eq2.4}
\alpha=[a_0;a_1,a_2,a_3,\ldots]=a_0+\frac{1}{a_1+\frac{1}{a_2+\frac{1}{a_3+\cdots}}},
\end{equation}
where $a_0\ge0$ and $a_i\ge1$, $i=1,2,3,\ldots,$ are integers.
The rational numbers
\begin{equation}\label{eq2.5}
\frac{p_k}{q_k}=\frac{p_k(\alpha)}{q_k(\alpha)}=[a_0;a_1,\ldots,a_k],
\quad
k=0,1,2,3,\ldots,
\end{equation}
are the $k$-convergents of~$\alpha$.
It is well known that they give rise to the best rational approximations of the irrational number~$\alpha$, and we have
\begin{equation}\label{eq2.6}
\frac{p_0}{q_0}
<\frac{p_2}{q_2}
<\frac{p_4}{q_4}
<\ldots
<\alpha
<\ldots
<\frac{p_5}{q_5}
<\frac{p_3}{q_3}
<\frac{p_1}{q_1}.
\end{equation}

Let $\Vert y\Vert$ denote the distance of a real number $y$ from the nearest integer.
We shall make use of the fact that for an irrational number $\alpha$ the sequence
\begin{displaymath}
\min_{1\le k\le n}\Vert k\alpha\Vert,
\quad
n=1,2,3,\ldots,
\end{displaymath}
is well described by the continued fraction expansion of~$\alpha$.

For every $k=0,1,2,3,\ldots,$ we have
\begin{equation}\label{eq2.7}
\Vert q\alpha\Vert\ge\Vert q_k\alpha\Vert,
\quad
1\le q<q_{k+1},
\end{equation}
\begin{displaymath}
\Vert q_{k+1}\alpha\Vert<\Vert q_k\alpha\Vert,
\end{displaymath}
as well as
\begin{equation}\label{eq2.8}
\frac{1}{q_{k+1}+q_k}\le\Vert q_k\alpha\Vert\le\frac{1}{q_{k+1}}.
\end{equation}
Indeed, the sequences $p_k$ and $q_k$, $k=0,1,2,3,\ldots,$ are given by the initial values
\begin{displaymath}
p_0=a_0,
\quad
p_1=a_1a_0+1,
\quad
q_0=1,
\quad
q_1=a_1,
\end{displaymath}
and the recurrence relations
\begin{equation}\label{eq2.9}
p_{k+1}=a_{k+1}p_k+p_{k-1},
\quad
q_{k+1}=a_{k+1}q_k+q_{k-1},
\quad
k\ge1.
\end{equation}
We also have
\begin{displaymath}
p_{k-1}q_k-q_{k-1}p_k=(-1)^k,
\quad
k\ge1.
\end{displaymath}
On the other hand, using \eqref{eq2.6} and \eqref{eq2.9}, it is easy to show that
\begin{equation}\label{eq2.10}
\Vert q_{k+1}\alpha\Vert+a_{k+1}\Vert q_k\alpha\Vert=\Vert q_{k-1}\alpha\Vert.
\end{equation}

We need the following result.

\begin{untheorem}[$3$-distance theorem]
Consider the $n+1$ numbers $0,\alpha,2\alpha,3\alpha,\ldots,n\alpha$ modulo~$1$ in the unit torus/circle $[0,1)$, leading to an $(n+1)$-partition.
This partition exhibits at most $3$ different distances between consecutive points.
Furthermore, every positive integer $n$ can be expressed uniquely in the form
\begin{displaymath}
n=\mu q_k+q_{k-1}+r,
\quad
\mbox{with $1\le \mu\le a_{k+1}$ and $0\le r<q_k$},
\end{displaymath}
in terms of the continued fraction \eqref{eq2.4} of~$\alpha$ and its convergents \eqref{eq2.5}, with the convention that $q_{-1}=0$.
Then
\begin{itemize}
\item[(i)] the distance $\Vert q_k\alpha\Vert$ shows up precisely $n+1-q_k$ times;
\item[(ii)] the distance $\Vert q_{k-1}\alpha\Vert-\mu\Vert q_k\alpha\Vert$ shows up precisely $r+1$ times; and
\item[(iii)] the distance $\Vert q_{k-1}\alpha\Vert-(\mu-1)\Vert q_k\alpha\Vert$ shows up precisely $q_k-r-1$ times.
\end{itemize}
\end{untheorem}

This surprising geometric fact, formulated as a conjecture by Steinhaus, has many proofs, by S\'{o}s~\cite{So1,So2}, Swierczkowski~\cite{Sw},
Sur\'{a}nyi~\cite{Su}, Halton~\cite{H} and Slater~\cite{Sl}, with others published more recently.

%
%

\section{Proof of Theorem~\ref{thm2}}\label{sec3}

Given an integer $k\ge1$, let $\AAA_k(\alpha)$ denote the partition of the unit torus/circle $[0,1)$ with $q_{k+1}=q_{k+1}(\alpha)$ division points
$\{q\alpha\}$, $-1\le q\le q_{k+1}-2$, where $\{x\}$ denotes the fractional part of a real number~$x$.
Note that the choices $q=-1,0$ in $\{q\alpha\}$ represent the \textit{dangerous} endpoints of the special intervals $[0,1-\alpha)$ and $[1-\alpha,1)$ in \eqref{eq2.1}.
These are the two \textit{singularities} of the interval exchange transformation $T$ restricted to the interval $0\le x<1$, in the sense that both $0$ and $T(1-\alpha)=1$ represent the split singularity of the L-surface.

A consequence of the special choice $n=q_{k+1}-1$ is that the $3$-distance theorem simplifies to a $2$-distance theorem.
This in turn leads to some very useful information concerning the distances between the consecutive points of the $q_{k+1}$-partition $\AAA_k(\alpha)$ of the unit torus/circle $[0,1)$.
Indeed, using the second recurrence relation in \eqref{eq2.9}, we have
\begin{displaymath}
n=q_{k+1}-1=a_{k+1}q_k+q_{k-1}-1=\mu q_k+q_{k-1}+r,
\end{displaymath}
with $\mu=a_{k+1}-1$ and $r=q_k-1$.
Since $q_k-r-1=0$, it follows from the $3$-distance theorem that there are only two distances
\begin{equation}\label{eq3.1}
\Vert q_k\alpha\Vert
\quad\mbox{and}\quad
\Vert q_{k-1}\alpha\Vert-(a_{k+1}-1)\Vert q_k\alpha\Vert=\Vert q_{k+1}\alpha\Vert+\Vert q_k\alpha\Vert,
\end{equation}
in view of \eqref{eq2.10}.

It follows immediately from \eqref{eq2.7} that one of the neighbors of $0$ in the partition $\AAA_k(\alpha)$ is $\{q_k\alpha\}$ which clearly has distance
$\Vert q_k\alpha\Vert$ from $0$ in the unit torus/circle.
Since $\alpha$ is irrational, the other neighbor of $0$ in the partition $\AAA_k(\alpha)$ must have distance
$\Vert q_{k+1}\alpha\Vert+\Vert q_k\alpha\Vert$ from $0$ in the unit torus/circle.
Simple calculation then shows that it is $\{((a_{k+1}-1)q_k+q_{k-1})\alpha\}$.
Thus the two neighbors
\begin{displaymath}
\{q_k\alpha\}
\quad\mbox{and}\quad
\{((a_{k+1}-1)q_k+q_{k-1})\alpha\}
\end{displaymath}
of $0$ in the partition $\AAA_k(\alpha)$ exhibit the two gaps in \eqref{eq3.1} in some order.
Similarly, the two neighbors
\begin{displaymath}
\{(q_k-1)\alpha\}
\quad\mbox{and}\quad
\{((a_{k+1}-1)q_k+q_{k-1}-1)\alpha\}
\end{displaymath}
of $1-\alpha=\{-\alpha\}$ in the partition $\AAA_k(\alpha)$ exhibit the same two gaps in \eqref{eq3.1} in the same order.

The union of the left and right neighborhoods of $0$ in the partition $\AAA_k(\alpha)$ has the form
\begin{equation}\label{eq3.2}
B(0)=(-d^{\ast},d^{\ast\ast}),
\end{equation}
and the union of the left and right neighborhoods of $1-\alpha =\{-\alpha\}$ in the partition $\AAA_k(\alpha)$ has a similar form
\begin{equation}\label{eq3.3}
B(-1)=(1-\alpha-d^{\ast},1-\alpha+d^{\ast\ast}),
\end{equation}
due to the same order, where
\begin{equation}\label{eq3.4}
\{d^{\ast},d^{\ast\ast}\}=\{\Vert q_k\alpha\Vert,\Vert q_{k+1}\alpha\Vert+\Vert q_k\alpha\Vert\},
\end{equation}
but we have not specified which one is which.
We refer to $B(0)$ and $B(-1)$ as the buffer zones of the singularities $0$ and $1-\alpha$ respectively in the partition $\AAA_k(\alpha)$.

We consider the special intervals
\begin{equation}\label{eq3.5}
J_k(q)=J(\alpha;k;q)=[\{q\alpha\}-d^{\ast\ast},\{q\alpha\}+d^{\ast}),
\quad
1\le q\le q_{k+1}-2.
\end{equation}
Note that these short special intervals have three crucial properties: 

(i) They completely cover the two long intervals $(0,1-\alpha)$ and $(1-\alpha,1)$.

(ii) They avoid the singularities $0,1$ and $1-\alpha$, in view of \eqref{eq3.2}--\eqref{eq3.4}.

(iii) Any two intervals in $(0,1-\alpha)$ or in $(1-\alpha,1)$ arising from neighboring partition points exhibit \textit{substantial overlapping}.
More precisely, if $1\le q',q''\le q_{k+1}-2$ are two integers such that $\{q'\alpha\}$ and $\{q''\alpha\}$ are neighboring points in the partition
$\AAA_k(\alpha)$, so that both points are in the interval $(0,1-\alpha )$ or both points are in the interval $(1-\alpha ,1)$, then
\begin{equation}\label{eq3.6}
\length(J_k(q')\cap J_k(q''))\ge\min\{d^{\ast},d^{\ast\ast}\}=\Vert q_k\alpha\Vert.
\end{equation}
Since the length of $J_k(q)$ is $2\Vert q_k\alpha\Vert+\Vert q_{k+1}\alpha\Vert$, the trivial upper bound
\begin{displaymath}
2\Vert q_k\alpha\Vert+\Vert q_{k+1}\alpha\Vert<3\Vert q_k\alpha\Vert
\end{displaymath}
and \eqref{eq3.6} together justify the term \textit{substantial overlapping}.

Since $T$ acts on the interval/circle $[0,3)$, for every interval $J_k(q)$, $1\le q\le q_{k+1}-2$, given by \eqref{eq3.5}, we define its \textit{$3$-copy extension} $J_k(q;3)$ by
\begin{equation}\label{eq3.7}
J_k(q;3)=J_k(q)\cup(1+J_k(q))\cup(2+J_k(q))\subset[0,3),
\end{equation}
a union of $J_k(q)$ with two of its translates.

After our preparation, we are now ready to study an orbit.
Let $\LLL_{\alpha}(S;t)$, $t\ge0$, be a parametrized half-infinite geodesic with initial point $S$ and slope $1/\alpha$, under the usual arc-length parametrization.

Let $M$ be \textit{large}, and consider the initial segment $\LLL_{\alpha}(S;t)$, $0\le t\le M$, of length~$M$, which we denote by $(\LLL_{\alpha};M)$.
Suppose that
\begin{equation}\label{eq3.8}
0\le t_1<t_2<t_3<\ldots<t_m\le M,
\end{equation}
where
\begin{equation}\label{eq3.9}
t_{i+1}-t_i=\sqrt{1+\alpha^2},
\quad
1\le i\le m-1,
\end{equation}
is the sequence of time instances $t$ when the initial segment $\LLL_{\alpha}(S;t)$, $0\le t\le M$, intersects the union $h_1\cup h_2\cup h_3=[0,3)$ of the $3$ horizontal edges of the L-surface in Figure~2.1.
For notational simplicity let
\begin{equation}\label{eq3.10}
y_i=\LLL_{\alpha}(S;t_i)\in[0,3),
\quad
1\le i\le m,
\end{equation}
denote these intersection points.

Using the interval exchange transformation $T=T_{\alpha}:[0,3)\to[0,3)$, we see that any two time-consecutive intersection points are governed by the simple relation
\begin{equation}\label{eq3.11}
T(y_i)=y_{i+1},
\quad
1\le i\le m-1.
\end{equation}

\begin{lem}\label{lem30}
Suppose that $J_k(\ell_1)$ is a special interval of the form \eqref{eq3.5}, and there exists $r(\ell_1)\in\{0,1,2\}$ such that
\begin{equation}\label{eq3.12}
\{y_i:1\le i\le m\}\cap(r(\ell_1)+J_k(\ell_1))=\emptyset,
\end{equation}
where $\{y_i:1\le i\le m\}$ is the set of intersection points defined in \eqref{eq3.10}.
Then for every integer $1-\ell_1\le h\le q_{k+1}-2-\ell_1$, we have
\begin{equation}\label{eq3.13}
\{y_i:q_{k+1}\le i\le m-q_{k+1}\}\cap T^h(r(\ell_1)+J_k(\ell_1))=\emptyset.
\end{equation}
\end{lem}

\begin{proof}
Since $q_{k+1}\ge1$ and $m-q_{k+1}\le m$, it follows trivially from \eqref{eq3.12} that \eqref{eq3.13} holds for $h=0$.

Combining \eqref{eq3.11} and \eqref{eq3.12}, we see that
\begin{displaymath}
\{y_i:2\le i\le m\}\cap T(r(\ell_1)+J_k(\ell_1))=\emptyset.
\end{displaymath}
Iterating this argument, we see that for every integer $1\le h\le q_{k+1}-2-\ell_1$, we have 
\begin{displaymath}
\{y_i:1+h\le i\le m\}\cap T^h(r(\ell_1)+J_k(\ell_1))=\emptyset.
\end{displaymath}
Since $q_{k+1}\ge q_{k+1}-1-\ell_1\ge h+1$ and $m-q_{k+1}\le m$, it follows that \eqref{eq3.13} holds for every integer $1\le h\le q_{k+1}-2-\ell_1$.

For every negative integer $1-\ell_1\le h\le-1$, using the inverse transformation $T^{-1}$, combining \eqref{eq3.11} and \eqref{eq3.12}, and iterating, we have
\begin{displaymath}
\{y_i:1\le i\le m+h\}\cap T^h(r(\ell_1)+J_k(\ell_1))=\emptyset.
\end{displaymath}
Since $q_{k+1}\ge1$ and $m+h\ge m+1-\ell_1\ge m+1-q_{k+1}\ge m-q_{k+1}$, it follows that \eqref{eq3.13} holds for every integer $1-\ell_1\le h\le-1$.

The proof of the lemma is now complete.
\end{proof}

\begin{remark}
We often refer to the deduction of \eqref{eq3.13} from \eqref{eq3.12} as a $T$-power extension argument.
\end{remark}

For notational convenience, for every integer $1-\ell_1\le h\le q_{k+1}-2-\ell_1$, we write
\begin{equation}\label{eq3.14}
T^h(r(\ell_1)+J_k(\ell_1))=r(\ell_1+h)+J_k(\ell_1+h).
\end{equation}
Note that \eqref{eq3.14} defines $r(q)$ for every integer $1\le q\le q_{k+1}-2$.
Furthermore, combining \eqref{eq3.13} and \eqref{eq3.14}, we have
\begin{equation}\label{eq3.15}
\{y_i:q_{k+1}\le i\le m-q_{k+1}\}\cap(r(q)+J_k(q))=\emptyset
\end{equation}
for every integer $1\le q\le q_{k+1}-2$.

Suppose that  $I_0\subset[0,3)$ is $(\LLL_\alpha;M)$-free, so that
\begin{displaymath}
\{y_i:1\le i\le m\}\cap I_0=\emptyset,
\end{displaymath}
where $\{y_i:1\le i\le m\}$ is the set of intersection points defined in \eqref{eq3.10}.
Let $k$ be an integer, and suppose that $J_k(\ell_1)$ is a special interval of the form \eqref{eq3.5}, and there exists $r(\ell_1)\in\{0,1,2\}$ such that
\begin{equation}\label{eq3.16}
r(\ell_1)+J_k(\ell_1)\subset I_0.
\end{equation}
Then \eqref{eq3.12} holds.

\begin{remark}
We shall later choose an \textit{optimal} value of $k$ for which \eqref{eq3.16} holds.
\end{remark}

We distinguish a few cases according to the special relations between various sets of intersection points and various special intervals.
We take advantage of the substantial overlapping of the short special intervals $J_k(q)$ defined by \eqref{eq3.5}.

Recall that if $1\le q',q''\le q_{k+1}-2$ are two integers such that $\{q'\alpha\}$ and $\{q''\alpha\}$ are neighboring points in the partition
$\AAA_k(\alpha)$, so that both points are in the interval $(0,1-\alpha )$ or both points are in the interval $(1-\alpha ,1)$, then combining \eqref{eq2.8} and \eqref{eq3.6}, we have
\begin{displaymath}
\length(J_k(q')\cap J_k(q''))
\ge\Vert q_k\alpha\Vert
\ge\frac{1}{q_{k+1}+q_k}
>\frac{1}{2q_{k+1}}.
\end{displaymath}
Recall also from \eqref{eq2.8}, \eqref{eq3.4} and \eqref{eq3.5} that
\begin{displaymath}
\length(J_k(q))
=2\Vert q_k\alpha\Vert+\Vert q_{k+1}\alpha\Vert
<3\Vert q_k\alpha\Vert
\le\frac{3}{q_{k+1}},
\end{displaymath}
so that
\begin{equation}\label{eq3.17}
\length(J_{k+8}(q))<\frac{3}{q_{k+9}}.
\end{equation}
On the other hand, a trivial deduction from \eqref{eq2.9} gives
\begin{displaymath}
q_{k+2}
=a_{k+2}q_{k+1}+q_k
\ge q_{k+1}+q_k
=a_{k+1}q_k+q_{k-1}+q_k
\ge2q_k,
\end{displaymath}
so that iterating this a few times, we conclude that
\begin{equation}\label{eq3.18}
q_{k+9}\ge2q_{k+7}\ge4q_{k+5}\ge8q_{k+3}\ge16q_{k+1}.
\end{equation}
Combining \eqref{eq3.17} and \eqref{eq3.18}, we conclude that the intersection $J_k(q')\cap J_k(q'')$ must contain a special interval of the type
$J_{k+8}(q)$ for some $1\le q\le q_{k+9}-2$.
We split the argument into two complementary cases.

\begin{case1a}
The following \textit{intersection property} holds.
For every
\begin{displaymath}
J_{k+8}(\ell)\subset J_k(q)
\quad\mbox{and}\quad
r\in\{0,1,2\},
\end{displaymath}
with $r\ne r(q)$ given by \eqref{eq3.14}, we have
\begin{equation}\label{eq3.19}
\{y_i:q_{k+1}\le i\le m-q_{k+1}\}\cap(r+J_{k+8}(\ell))\ne\emptyset.
\end{equation}
\end{case1a}

\begin{lem}\label{lem31}
Case~1A is impossible.
\end{lem}

\begin{case1b}
There exist
\begin{displaymath}
J_{k+8}(\ell_2)\subset J_k(q^{+})
\quad\mbox{and}\quad
r_1\in\{0,1,2\},
\end{displaymath}
with $r_1\ne r(q^{+})$ given by \eqref{eq3.14}, such that
\begin{equation}\label{eq3.20}
\{y_i:q_{k+1}\le i\le m-q_{k+1}\}\cap(r_1+J_{k+8}(\ell_2))=\emptyset.
\end{equation}
\end{case1b}

Since $J_{k+8}(\ell_2)\subset J_k(q^{+})$, it clearly follows from \eqref{eq3.15} that
\begin{equation}\label{eq3.21}
\{y_i:q_{k+1}\le i\le m-q_{k+1}\}\cap(r(q^{+})+J_{k+8}(\ell_2))=\emptyset.
\end{equation}
Since \eqref{eq3.20} and \eqref{eq3.21} are analogs of \eqref{eq3.12}, the $T$-power expansion argument in the proof of Lemma~\ref{lem30} shows that for every integer $1-\ell_2\le h\le q_{k+9}-2-\ell_2$, we have
\begin{align}
&
\{y_i:q_{k+1}+q_{k+9}\le i\le m-q_{k+1}-q_{k+9}\}\cap T^h(r_2+J_{k+8}(\ell_2))=\emptyset,
\label{eq3.22}
\\
&
\{y_i:q_{k+1}+q_{k+9}\le i\le m-q_{k+1}-q_{k+9}\}\cap T^h(r(q^+)+J_{k+8}(\ell_2))=\emptyset.
\label{eq3.23}
\end{align}
For notational convenience, for every integer $1-\ell_2\le h\le q_{k+9}-2-\ell_2$, we write
\begin{align}
T^h(r(q^{+})+J_{k+8}(\ell_2))
&
=r^{\star}(\ell_2+h)+J_{k+8}(\ell_2+h),
\label{eq3.24}
\\
T^h(r_2+J_{k+8}(\ell_2))
&
=r^{\star\star}(\ell_2+h)+J_{k+8}(\ell_2+h).
\label{eq3.25}
\end{align}
Then combining \eqref{eq3.22}--\eqref{eq3.25}, we have
\begin{align}
&
\{y_i:q_{k+1}+q_{k+9}\le i\le m-q_{k+1}-q_{k+9}\}\cap(r^{\star}(q)+J_{k+8}(q)=\emptyset,
\label{eq3.26}
\\
&
\{y_i:q_{k+1}+q_{k+9}\le i\le m-q_{k+1}-q_{k+9}\}\cap(r^{\star\star}(q)+J_{k+8}(q)=\emptyset,
\label{eq3.27}
\end{align}
for every integer $1\le q\le q_{k+9}-2$.
Clearly
\begin{displaymath}
r^{\star}(q)\ne r^{\star\star}(q),
\quad
1\le q\le q_{k+9}-2.
\end{displaymath}
We now split Case~1B into two complementary cases.

\begin{case2a}
The following \textit{intersection property} holds.
For every
\begin{displaymath}
J_{k+16}(\ell)\subset J_{k+8}(q)
\quad\mbox{and}\quad
r\in\{0,1,2\},
\end{displaymath}
with $r\ne r^{\star}(q),r^{\star\star}(q)$ given by \eqref{eq3.24} and \eqref{eq3.25}, we have
\begin{displaymath}
\{y_i:q_{k+1}+q_{k+9}\le i\le m-q_{k+1}-q_{k+9}\}\cap(r+J_{k+16}(\ell))\ne\emptyset.
\end{displaymath}
\end{case2a}

\begin{lem}\label{lem32a}
Case~2A is impossible.
\end{lem}

\begin{case2b}
There exist
\begin{displaymath}
J_{k+16}(\ell_3)\subset J_{k+8}(q^{++})
\quad\mbox{and}\quad
r_2\in\{0,1,2\},
\end{displaymath}
with $r_2\ne r^{\star}(q^{++}),r^{\star\star}(q^{++})$ given by \eqref{eq3.24} and \eqref{eq3.25}, such that
\begin{equation}\label{eq3.28}
\{y_i:q_{k+1}+q_{k+9}\le i\le m-q_{k+1}-q_{k+9}\}\cap(r_2+J_{k+16}(\ell_3))=\emptyset.
\end{equation}
\end{case2b}

\begin{lem}\label{lem32b}
If Case~2B holds, then
\begin{equation}\label{eq3.29}
m\le 2q_{k+1}+2q_{k+9}+2q_{k+17}+6.
\end{equation}
\end{lem}

\begin{proof}[Proof of Theorem~\ref{thm2}]
Suppose that $I_0\subset[0,3)$ is $(\LLL_\alpha;M)$-free, so that
\begin{displaymath}
\{y_i:1\le i\le m\}\cap I_0=\emptyset.
\end{displaymath}
Let $J_k(\ell_1)=J(\alpha;k;\ell_1)$ be the \textit{longest} special interval of the form \eqref{eq3.5} such that
\begin{displaymath}
r(\ell_1)+J_k(\ell_1)\subset I_0
\quad\mbox{for some $\ell_1$ and $r(\ell_1)\in\{0,1,2\}$}.
\end{displaymath}
Then
\begin{equation}\label{eq3.30}
\length(I_0)<4(\Vert q_{k-1}\alpha\Vert+\Vert q_{k}\alpha\Vert),
\end{equation}
since otherwise there exists a longer special interval
\begin{displaymath}
r(\ell)+J_{k-1}(\ell)\subset I_0
\quad\mbox{for some $\ell$ and $r(\ell)\in\{0,1,2\}$},
\end{displaymath}
a contradiction.
Combining \eqref{eq3.30} with \eqref{eq2.8}, we have
\begin{equation}\label{eq3.31}
\length(I_0)<\frac{4}{q_{k}}+\frac{4}{q_{k+1}}<\frac{8}{q_{k}}.
\end{equation}

On the other hand, it follows from \eqref{eq3.8} and \eqref{eq3.9} that
\begin{equation}\label{eq3.32}
M\le(m+1)\sqrt{1+\alpha^2}.
\end{equation}
Also, in view of Lemmas \ref{lem31}--\ref{lem32b}, it is clear that the bound \eqref{eq3.29} holds.
Finally, recall that $0<\alpha<1$ and
\begin{displaymath}
\alpha=[a_1,a_2,a_3,\ldots]=\frac{1}{a_1+\frac{1}{a_2+\frac{1}{a_3+\cdots}}}
\end{displaymath}
is badly approximable, so there exists a constant $A$ such that the continued fraction digits $a_i\le A$ for every $i=1,2,3,\ldots.$
It follows from \eqref{eq2.9} that
\begin{equation}\label{eq3.33}
q_{k+1}<q_{k+9}<q_{k+17}\le(A+1)^{17}q_k.
\end{equation}
Combining \eqref{eq3.29}, \eqref{eq3.32} and \eqref{eq3.33}, we see that
\begin{equation}\label{eq3.34}
M\le(2q_{k+1}+2q_{k+9}+2q_{k+17}+7)\sqrt{1+\alpha^2}
<7(A+1)^{17}q_k\sqrt{2}.
\end{equation}

It now follows from \eqref{eq3.31} and \eqref{eq3.34} that a geodesic segment $\LLL_{\alpha}(S;t)$, $0\le t\le M$, of length $M=7(A+1)^{17}q_k\sqrt{2}$ must intersect every subinterval $I$ of $h_1\cup h_2\cup h_3$ with $\length(I)=8/q_{k}$.
Since the product $M\length(I)=56(A+1)^{17}\sqrt{2}$ is a constant independent of~$k$, this establishes superdensity of the half-infinite geodesic.
\end{proof}

%
%

\section{Proof of Lemmas \ref{lem31}--\ref{lem32b}}\label{sec4}

Before we present the proof of our main lemmas, we begin by investigating a simple situation which serves to illustrate our method.

\begin{case0}
There exist integers $r^*(\ell_1)$ and $r^{**}(\ell_1)$ such that

(1) $r(\ell_1),r^*(\ell_1),r^{**}(\ell_1)$ form a permutation of $0,1,2$;

(2) $\{y_i:1\le i\le m\}\cap(r^*(\ell_1)+J_k(\ell_1))=\emptyset$; and

(3) $\{y_i:1\le i\le m\}\cap(r^{**}(\ell_1)+J_k(\ell_1))=\emptyset$.
\end{case0}

\begin{lem}\label{lem41}
If the Simple Case holds, then $m\le 2q_{k+1}+6$.
\end{lem}

\begin{proof}
For notational simplicity, we write
\begin{equation}\label{eq4.1}
Q(k;m)=\{y_i:q_{k+1}\le i\le m-q_{k+1}\}.
\end{equation}
Since the properties (2) and (3) in the Simple Case are analogs of \eqref{eq3.12}, we can repeat the $T$-power extension argument in Lemma~\ref{lem30} and conclude that for every integer $1-\ell_1\le h\le q_{k+1}-2-\ell_1$, we have
\begin{align}
&
Q(k;m)\cap T^h(r^*(\ell_1)+J_k(\ell_1))=\emptyset,
\label{eq4.2}
\\
&
Q(k;m)\cap T^h(r^{**}(\ell_1)+J_k(\ell_1))=\emptyset.
\label{eq4.3}
\end{align}
As in \eqref{eq3.14}, we write
\begin{align}
T^h(r^*(\ell_1)+J_k(\ell_1))
&
=r^*(\ell_1+h)+J_k(\ell_1+h),
\label{eq4.4}
\\
T^h(r^{**}(\ell_1)+J_k(\ell_1))
&
=r^{**}(\ell_1+h)+J_k(\ell_1+h),
\label{eq4.5}
\end{align}
for every integer $1-\ell_1\le h\le q_{k+1}-2-\ell_1$.
Note that \eqref{eq4.4} and \eqref{eq4.5} define $r^*(q)$ and $r^{**}(q)$ respectively for every integer $1\le q\le q_{k+1}-2$.
Recall next that $r(q)$ is defined by \eqref{eq3.14}.
Indeed, using the notation \eqref{eq3.14}, it is easy to check that the assertion \eqref{eq3.13} for every integer $1-\ell_1\le h\le q_{k+1}-2-\ell_1$ implies that
\begin{equation}\label{eq4.6}
Q(k;m)\cap(r(\ell_1+h)+J_k(\ell_1+h))=\emptyset
\end{equation}
for every integer $1-\ell_1\le h\le q_{k+1}-2-\ell_1$.

Combining \eqref{eq4.2}--\eqref{eq4.6} for every integer $1-\ell_1\le h\le q_{k+1}-2-\ell_1$, we deduce that for every integer $1\le q\le q_{k+1}-2$, we have
\begin{align}
&
Q(k;m)\cap(r^*(q)+J_k(q))=\emptyset,
\label{eq4.7}
\\
&
Q(k;m)\cap(r^{**}(q)+J_k(q))=\emptyset,
\label{eq4.8}
\\
&
Q(k;m)\cap(r(q)+J_k(q))=\emptyset.
\label{eq4.9}
\end{align}
Also, in view of the property (1) in the Simple Case, it is clear that $r(q),r^*(q),r^{**}(q)$ form a permutation of $0,1,2$ for every integer
$1\le q\le q_{k+1}-2$.

By definition, the $3$-copy extensions $J_k(q;3)$, $1\le q\le q_{k+1}-2$, give rise to $6$ continuous chains of overlapping intervals in the torus/circle $[0,3)$ such that the $6$ chains completely cover the $6$ intervals
\begin{displaymath}
(0,1-\alpha),
\quad
(1-\alpha,1),
\quad
(1,2-\alpha),
\quad
(2-\alpha,2),
\quad
(2,3-\alpha),
\quad
(3-\alpha,3),
\end{displaymath}
and there are only $6$ points in $[0,3)$ that are not covered by the $6$ chains, namely
\begin{displaymath}
0,
\quad
1-\alpha,
\quad
1,
\quad
2-\alpha,
\quad
2,
\quad3-\alpha.
\end{displaymath}

Combining \eqref{eq4.7}--\eqref{eq4.9} for every integer $1\le q\le q_{k+1}-2$, we deduce that the set $Q(k;m)$ is not covered by the $6$ chains. 
Indeed, if $m\ge2q_{k+1}+7$, then the set \eqref{eq4.1} has at least $7$ distinct elements, which is more than~$6$, giving rise to a contradiction. 
We conclude therefore that, under the conditions of the Simple Case, we must have $m\le2q_{k+1}+6$, and this completes the proof.
\end{proof}

\begin{proof}[Proof of Lemma~\ref{lem31}]
Again, for notational simplicity, we use \eqref{eq4.1}.

Our first step is to prove that, under the condition of Case~1A, any two neighboring $3$-copy extensions $J_k(q';3)$ and $J_k(q'';3)$ are \textit{synchronized} in the following precise sense:
For each $r\in\{0,1,2\}$, we have
\begin{equation}\label{eq4.10}
Q(k;m)\cap(r+J_{k}(q'))=\emptyset
\quad\mbox{if and only if}\quad
Q(k;m)\cap(r+J_{k}(q''))=\emptyset.
\end{equation}
To establish this, we consider two cases.

Suppose first that $r=r(q')$.
Using the notation \eqref{eq3.14}, it follows from \eqref{eq3.13} that
\begin{equation}\label{eq4.11}
Q(k;m)\cap(r(q')+J_k(q'))=\emptyset.
\end{equation}
Assume on the contrary that
\begin{equation}\label{eq4.12}
Q(k;m)\cap(r(q')+J_k(q''))\ne\emptyset.
\end{equation}
Then it follows from \eqref{eq3.15} that $r(q'')\ne r(q')$.
On the other hand, we know that the intersection $J_k(q')\cap J_k(q'')$ must contain a special interval of the type $J_{k+8}(j_0)$ for some~$j_0$, so
\begin{equation}\label{eq4.13}
r(q')+J_{k+8}(j_0)\subset(r(q')+J_k(q'))\cap(r(q')+J_k(q'')).
\end{equation}
Since $r(q')\ne r(q'')$ and $J_{k+8}(j_0)\subset J_k(q'')$, the condition of Case~1A is satisfied with $q=q''$, and \eqref{eq3.19} becomes
\begin{equation}\label{eq4.14}
Q(k;m)\cap(r(q')+J_{k+8}(j_0))\ne\emptyset.
\end{equation}
But \eqref{eq4.11} and \eqref{eq4.13} contradict \eqref{eq4.14}, and so \eqref{eq4.12} fails.
Thus our claim \eqref{eq4.10} holds in this case.

Suppose next that $r\ne r(q')$.
As before, we use the fact that the intersection $J_k(q')\cap J_k(q'')$ must contain a special interval of the type $J_{k+8}(j_0)$ for some~$j_0$, so
\begin{equation}\label{eq4.15}
r+J_{k+8}(j_0)\subset(r+J_k(q'))\cap (r+J_k(q'')).
\end{equation}
Since $r\ne r(q')$, the condition of Case~1A is satisfied with $q=q'$, and \eqref{eq3.19} becomes
\begin{equation}\label{eq4.16}
Q(k;m)\cap(r+J_{k+8}(j_0))\ne\emptyset.
\end{equation}
It now follows from \eqref{eq4.15} and \eqref{eq4.16} that
\begin{displaymath}
Q(k;m)\cap(r+J_{k}(q'))\ne\emptyset
\quad\mbox{and}\quad
Q(k;m)\cap(r+J_k(q''))\ne\emptyset,
\end{displaymath}
so that our claim \eqref{eq4.10} holds also in this case.

By definition, the $3$-copy extensions $J_k(q;3)$, $1\le q\le q_{k+1}-2$, give rise to $6$ continuous chains of overlapping intervals in the torus/circle $[0,3)$ such that the $6$ chains completely cover the $6$ intervals
\begin{equation}\label{eq4.17}
(0,1-\alpha),
\quad
(1-\alpha,1),
\quad
(1,2-\alpha),
\quad
(2-\alpha,2),
\quad
(2,3-\alpha),
\quad
(3-\alpha,3).
\end{equation}
The synchronization property we have just established now implies that each of the $6$ long special intervals in \eqref{eq4.17} satisfies one of the following two properties.
Either such a long special interval is \textit{disjoint} from the set $Q(k;m)$, or the set $Q(k;m)$ is \textit{dense} in such a long special interval, in the precise sense that every subinterval of length $1/q_{k+8}$ contains a point from the set $Q(k;m)$.

Moreover, it is not difficult to show that precisely $2$ of the $6$ long special intervals in \eqref{eq4.17} are disjoint from the set $Q(k;m)$.
To see this, choose two integers $q'$ and $q''$ satisfying $1\le q',q''\le q_{k+1}-2$ such that
\begin{equation}\label{eq4.18}
J_k(q')\subset(0,1-\alpha)
\quad\mbox{and}\quad
J_k(q'')\subset(1-\alpha,1).
\end{equation}
Then it follows from \eqref{eq3.13} and \eqref{eq3.14} that
\begin{equation}\label{eq4.19}
Q(k;m)\cap(r(q')+J_k(q'))=\emptyset
\quad\mbox{and}\quad
Q(k;m)\cap(r(q'')+J_k(q''))=\emptyset.
\end{equation}
Now write
\begin{equation}\label{eq4.20}
\III_1=r(q')+(0,1-\alpha)
\quad\mbox{and}\quad
\III_2=r(q'')+(1-\alpha,1).
\end{equation}
The synchronization property and \eqref{eq4.18}--\eqref{eq4.20} now imply that
\begin{displaymath}
Q(k;m)\cap\III_1=\emptyset
\quad\mbox{and}\quad
Q(k;m)\cap\III_2=\emptyset.
\end{displaymath}
Note that the union $\III_1\cup\III_2$ modulo~$1$ is precisely the unit interval $[0,1)$.

Now $\III_1$ and $\III_2$ are $2$ of the $6$ long special intervals in \eqref{eq4.17}.
Let $\III_j$, $3\le j\le 6$, denote the remaining long special intervals in \eqref{eq4.17}.
The condition of Case~1A now implies that these $4$ intervals are not disjoint from $Q(k;m)$, so that $Q(k;m)$ is dense in each of them.
Each $T$-image $T(\III_j)$, $j=1,2$, has at most $1$ common point with the set $Q(k;m)$.
This is a contradiction, since the union $T(\III_1)\cup T(\III_2)$ has a \textit{substantial} intersection with the union $\III_3\cup\III_4\cup\III_5\cup\III_6$, which implies that it must have a \textit{substantial} intersection with the set $Q(k;m)$, much more than at most $2$ elements.
Thus Case~1A is impossible, and this completes the proof.
\end{proof}

\begin{proof}[Proof of Lemma~\ref{lem32a}]
For notational simplicity, we write
\begin{equation}\label{eq4.21}
Q^*(k;m)=\{y_i:q_{k+1}+q_{k+9}\le i\le m-q_{k+1}-q_{k+9}\}.
\end{equation}

We can proceed along similar lines as in the first part of the proof of Lemma~\ref{lem31} for Case~1A, and show that any two neighboring $3$-copy extensions $J_{k+8}(q';3)$ and $J_{k+8}(q'';3)$ are \textit{synchronized} in the following precise sense:
For each $r\in\{0,1,2\}$, we have
\begin{displaymath}
Q^*(k;m)\cap(r+J_{k+8}(q'))=\emptyset
\quad\mbox{if and only if}\quad
Q^*(k;m)\cap(r+J_{k+8}(q''))=\emptyset.
\end{displaymath}

By definition, the $3$-copy extensions $J_{k+8}(q;3)$, $1\le q\le q_{k+9}-2$, give rise to $6$ continuous chains of overlapping intervals in the torus/circle $[0,3)$ such that the $6$ chains completely cover the $6$ intervals \eqref{eq4.17}.
The synchronization property now implies that each of the $6$ long special intervals in \eqref{eq4.17} satisfies one of the following two properties.
Either such a long special interval is \textit{disjoint} from the set $Q^*(k;m)$, or the set $Q^*(k;m)$ is \textit{dense} in such a long special interval, in the precise sense that every subinterval of length $1/q_{k+16}$ contains a point from the set $Q^*(k;m)$.

Moreover, it is not difficult to show that precisely $4$ of the $6$ long special intervals in \eqref{eq4.17} are disjoint from the set $Q^*(k;m)$.
To see this, choose two integers $q'$ and $q''$ satisfying $1\le q',q''\le q_{k+9}-2$ such that
\begin{equation}\label{eq4.22}
J_{k+8}(q')\subset(0,1-\alpha)
\quad\mbox{and}\quad
J_{k+8}(q'')\subset(1-\alpha,1).
\end{equation}
Then it follows from \eqref{eq3.26} and \eqref{eq3.27} that
\begin{align}
Q^*(k;m)\cap(r^{\star}(q')+J_{k+8}(q'))=\emptyset,
&\quad
Q^*(k;m)\cap(r^{\star}(q'')+J_{k+8}(q''))=\emptyset,
\label{eq4.23}
\\
Q^*(k;m)\cap(r^{\star\star}(q')+J_{k+8}(q'))=\emptyset,
&\quad
Q^*(k;m)\cap(r^{\star\star}(q'')+J_{k+8}(q''))=\emptyset.
\label{eq4.24}
\end{align}
Now write
\begin{align}
\III_1=r^{\star}(q')+(0,1-\alpha),
&\quad
\III_2=r^{\star}(q'')+(1-\alpha,1),
\label{eq4.25}
\\
\III_3=r^{\star\star}(q')+(0,1-\alpha),
&\quad
\III_4=r^{\star\star}(q'')+(1-\alpha,1).
\label{eq4.26}
\end{align}
The synchronization property and \eqref{eq4.22}--\eqref{eq4.26} now imply that
\begin{align}
Q^*(k;m)\cap\III_1=\emptyset,
&\quad
Q^*(k;m)\cap\III_2=\emptyset,
\nonumber
\\
Q^*(k;m)\cap\III_3=\emptyset,
&\quad
Q^*(k;m)\cap\III_4=\emptyset.
\nonumber
\end{align}
Note that the union $\III_1\cup\III_2\cup\III_3\cup\III_4$ modulo~$1$ is precisely the unit interval $[0,1)$ twice.

Now $\III_1,\III_2,\III_3,\III_4$ are $4$ of the $6$ long special intervals in \eqref{eq4.17}.
Let $\III_j$, $j=5,6$, denote the remaining long special intervals in \eqref{eq4.17}.
The condition of Case~2A now implies that these $2$ intervals are not disjoint from $Q^*(k;m)$, so that $Q^*(k;m)$ is dense in each of them.
Each $T$-image $T(\III_j)$, $j=1,\ldots,4$, has at most $1$ common point with the set $Q^*(k;m)$.
This is a contradiction, since the union
\begin{displaymath}
T(\III_1)\cup T(\III_2)\cup T(\III_3)\cup T(\III_4)
\end{displaymath}
has a \textit{substantial} intersection with the union $\III_5\cup\III_6$, which implies that it must have a \textit{substantial} intersection with the set $Q^*(k;m)$, much more than at most $4$ elements.
Thus Case~2A is impossible, and this completes the proof.
\end{proof}

\begin{proof}[Proof of Lemma~\ref{lem32b}]
For notational simplicity, we define $Q^*(k;m)$ by \eqref{eq4.21}, and write
\begin{displaymath}
Q^{**}(k;m)=\{y_i:q_{k+1}+q_{k+9}+q_{k+17}\le i\le m-q_{k+1}-q_{k+9}-q_{k+17}\}.
\end{displaymath}
Since \eqref{eq3.28} is an analog of \eqref{eq3.12}, we can repeat the $T$-power extension argument in Lemma~\ref{lem30} and conclude that for every integer $1-\ell_3\le h\le q_{k+17}-2-\ell_3$, we have
\begin{equation}\label{eq4.27}
Q^{**}(k;m)\cap T^h(r_2+J_{k+16}(\ell_3))=\emptyset.
\end{equation}
Since $J_{k+16}(\ell_3)\subset J_{k+8}(q^{++})$, it follows from \eqref{eq3.26} and \eqref{eq3.27} that
\begin{align}
&
Q^*(k;m)\cap(r^{\star}(q^{++})+J_{k+16}(\ell_3))=\emptyset,
\label{eq4.28}
\\
&
Q^*(k;m)\cap(r^{\star\star}(q^{++})+J_{k+16}(\ell_3))=\emptyset.
\label{eq4.29}
\end{align}
Next, note that \eqref{eq4.28} and \eqref{eq4.29} are also analogs of \eqref{eq3.12}, so again we can repeat the $T$-power extension argument in Lemma~\ref{lem30} and conclude that for every integer $1-\ell_3\le h\le q_{k+17}-2-\ell_3$, we have
\begin{align}
&
Q^{**}(k;m)\cap T^h(r^{\star}(q^{++})+J_{k+16}(\ell_3))=\emptyset,
\label{eq4.30}
\\
&
Q^{**}(k;m)\cap T^h(r^{\star\star}(q^{++})+J_{k+16}(\ell_3))=\emptyset.
\label{eq4.31}
\end{align}
Now, for every integer $1-\ell_3\le h\le q_{k+17}-2-\ell_3$, we write
\begin{align}
T^h(r^{\star}(q^{++})+J_{k+16}(\ell_3))
&
=r^{(0)}(\ell_3+h)+J_{k+16}(\ell_3+h),
\label{eq4.32}
\\
T^h(r^{\star\star}(q^{++})+J_{k+16}(\ell_3))
&
=r^{(1)}(\ell_3+h)+J_{k+16}(\ell_3+h),
\label{eq4.33}
\\
T^h(r_2+J_{k+16}(\ell_3))
&
=r^{(2)}(\ell_3+h)+J_{k+16}(\ell_3+h).
\label{eq4.34}
\end{align}
Clearly it follows from the assumption of Case~2B that $r^{(0)}(q),r^{(1)}(q),r^{(2)}(q)$ form a permutation of $0,1,2$ for every integer
$1\le q\le q_{k+17}-2$.
Combining \eqref{eq4.27} and \eqref{eq4.30}--\eqref{eq4.34}, we have
\begin{align}
&
Q^{**}(k;m)\cap(r^{(0)}(q)+J_{k+16}(q))=\emptyset,
\label{eq4.35}
\\
&
Q^{**}(k;m)\cap(r^{(1)}(q)+J_{k+16}(q))=\emptyset,
\label{eq4.36}
\\
&
Q^{**}(k;m)\cap(r^{(2)}(q)+J_{k+16}(q))=\emptyset,
\label{eq4.37}
\end{align}
for every integer $1\le q\le q_{k+17}-2$.

Note now that \eqref{eq4.35}--\eqref{eq4.37} are similar to \eqref{eq4.7}--\eqref{eq4.9} in the proof of Lemma~\ref{lem41}, so we now mimic the last part of that proof.

By definition, the $3$-copy extensions $J_{k+16}(q;3)$, $1\le q\le q_{k+17}-2$, give rise to $6$ continuous chains of overlapping intervals in the
torus/circle $[0,3)$ such that the $6$ chains completely cover the $6$ intervals
\begin{displaymath}
(0,1-\alpha),
\quad
(1-\alpha,1),
\quad
(1,2-\alpha),
\quad
(2-\alpha,2),
\quad
(2,3-\alpha),
\quad
(3-\alpha,3),
\end{displaymath}
and there are only $6$ points in $[0,3)$ that are not covered by the $6$ chains, namely
\begin{displaymath}
0,
\quad
1-\alpha,
\quad
1,
\quad
2-\alpha,
\quad
2,
\quad3-\alpha.
\end{displaymath}

Combining \eqref{eq4.35}--\eqref{eq4.37} for every integer $1\le q\le q_{k+17}-2$, we deduce that the set $Q^{**}(k;m)$ is not covered by the $6$ chains. 
Indeed, if $m\ge2q_{k+1}+2q_{k+9}+2q_{k+17}+7$, then the set $Q^{**}(k;m)$ has at least $7$ distinct elements, which is more than~$6$, giving rise to a contradiction. 
We conclude therefore that, under the conditions of Case~2B, we must have $m\le2q_{k+1}+2q_{k+9}+2q_{k+17}+6$, and this completes the proof.
\end{proof}

%
%

\section{Proof of Theorem~\ref{thm1}}\label{sec5}

Consider now an arbitrary finite polysquare surface~$\PPP$.
We are now in a position to complete the proof of Theorem~\ref{thm1}, and show that if the slope of a half-infinite geodesic on $\PPP$ is a badly approximable number, then the geodesic is superdense on~$\PPP$. 

The proof is a fairly straightforward adaptation of the proof of Theorem~\ref{thm2}, apart from the observation that the number of cases we need to consider is a function of the number of square faces of~$\PPP$, and so can be arbitrarily large.

Without loss of generality, we assume as before that the slope of the half-infinite geodesic is greater than~$1$.
Suppose that it has slope $1/\alpha$, where $0<\alpha<1$ is irrational.

Our first step is to generalize the interval exchange transformation $T=T_{\alpha}$ defined in Section~\ref{sec2}.
Suppose that the polysquare surface $\PPP$ has $s$ square faces.
Each square face has a top horizontal edge and a bottom horizontal edge.
Each top horizontal edge is identified with a unique bottom horizontal edge, and these give rise to $s$ horizontal edges $h_1,h_2,\ldots,h_s$.
We now identify these horizontal edges with unit intervals by making use of the correspondences
\begin{displaymath}
h_1=[0,1),
\quad
h_2=[1,2),
\quad
\ldots,
\quad
h_s=[s-1,s).
\end{displaymath}

We now consider the piecewise linear map $T=T_{\alpha}$ defined according to the analog of Figure~2.2 that corresponds to~$\PPP$.
More precisely, for each integer $1\le j\le s$, there exist unique integers $0\le j',j''\le s-1$ such that
\begin{displaymath}
T([j-1,j-\alpha))=[j'+\alpha,j'+1)
\quad\mbox{and}\quad
T([j-\alpha,j))=[j'',j''+\alpha).
\end{displaymath}
Indeed, it is not difficult to see that the map $T:[0,s)\to[0,s)$ is one-to-one and onto.
While $T$ acts on the longer interval $[0,s)$ instead of the unit interval $[0,1)$, if we consider $T$ modulo~$1$, then it acts simply as an $\alpha$-shift, or irrational rotation, in the unit interval.

As before, our main idea involves continued fractions, in particular, the special case of the $3$-distance theorem which becomes a $2$-distance theorem.
Indeed, we can repeat our discussion at the beginning of Section~\ref{sec3} \textit{verbatim} up to the end of the paragraph preceeding \eqref{eq3.7}.
In particular, for any integer $1\le q\le q_{k+1}-2$, we can define the special interval $J_k(q)$ as in \eqref{eq3.5}.
Then analogous to \eqref{eq3.7}, for any integer $1\le q\le q_{k+1}-2$, we define its $s$-copy extension $J_k(q;s)$ by
\begin{displaymath}
J_k(q;s)=J_k(q)\cup(1+J_k(q))\cup\ldots\cup((s-1)+J_k(q))\subset[0,s),
\end{displaymath}
a union of $J_k(q)$ with $s-1$ of its translates.

Let $\LLL_{\alpha}(t)=\LLL_{\alpha}(\PPP;S;t)$, $t\ge0$, be a parametrized half-infinite geodesic on $\PPP$ with initial point $S$ and slope $1/\alpha$, under the usual arc-length parametrization.

Let $M$ be \textit{large}, and consider the initial segment $\LLL_{\alpha}(t)$, $0\le t\le M$, of length~$M$, which we denote by $(\LLL_{\alpha};M)$.
Suppose that
\begin{equation}\label{eq5.1}
0\le t_1<t_2<t_3<\ldots<t_m\le M,
\end{equation}
where
\begin{equation}\label{eq5.2}
t_{i+1}-t_i=\sqrt{1+\alpha^2},
\quad
1\le i\le m-1,
\end{equation}
is the sequence of time instances $t$ when the initial segment $\LLL_{\alpha}(t)$, $0\le t\le M$, intersects the union $h_1\cup\ldots\cup h_s=[0,s)$ of the $s$ horizontal edges of the polysquare surface~$\PPP$.
For notational simplicity let
\begin{equation}\label{eq5.3}
y_i=\LLL_{\alpha}(t_i)\in[0,s),
\quad
1\le i\le m,
\end{equation}
denote these intersection points.

Using the interval exchange transformation $T=T_{\alpha}:[0,s)\to[0,s)$, we see that any two time-consecutive intersection points are governed by the simple relation
\begin{displaymath}
T(y_i)=y_{i+1},
\quad
1\le i\le m-1.
\end{displaymath}

We have the following analog of Lemma~\ref{lem30} which is easily established by the $T$-power extension argument.

\begin{lem}\label{lem50}
Suppose that $J_k(\ell_1)$ is a special interval of the form \eqref{eq3.5}, and there exists $r(\ell_1)\in\{0,1,\ldots,s-1\}$ such that
\begin{equation}\label{eq5.4}
\{y_i:1\le i\le m\}\cap(r(\ell_1)+J_k(\ell_1))=\emptyset,
\end{equation}
where $\{y_i:1\le i\le m\}$ is the set of intersection points defined in \eqref{eq5.3}.
Then for every integer $1-\ell_1\le h\le q_{k+1}-2-\ell_1$, we have
\begin{equation}\label{eq5.5}
\{y_i:q_{k+1}\le i\le m-q_{k+1}\}\cap T^h(r(\ell_1)+J_k(\ell_1))=\emptyset.
\end{equation}
\end{lem}

For notational convenience, for every integer $1-\ell_1\le h\le q_{k+1}-2-\ell_1$, we write
\begin{equation}\label{eq5.6}
T^h(r(\ell_1)+J_k(\ell_1))=r_1^{(0)}(\ell_1+h)+J_k(\ell_1+h).
\end{equation}
Note that \eqref{eq5.6} defines $r(q)$ for every integer $1\le q\le q_{k+1}-2$.
Furthermore, combining \eqref{eq5.5} and \eqref{eq5.6}, we have
\begin{equation}\label{eq5.7}
\{y_i:q_{k+1}\le i\le m-q_{k+1}\}\cap(r_1^{(0)}(q)+J_k(q))=\emptyset
\end{equation}
for every integer $1\le q\le q_{k+1}-2$.

Suppose that  $I_0\subset[0,s)$ is $(\LLL_\alpha;M)$-free, so that
\begin{displaymath}
\{y_i:1\le i\le m\}\cap I_0=\emptyset,
\end{displaymath}
where $\{y_i:1\le i\le m\}$ is the set of intersection points defined in \eqref{eq5.3}.
Let $k$ be an integer, and suppose that $J_k(\ell_1)$ is a special interval of the form \eqref{eq3.5}, and there exists $r(\ell_1)\in\{0,1,\ldots,s-1\}$ such that
\begin{equation}\label{eq5.8}
r(\ell_1)+J_k(\ell_1)\subset I_0.
\end{equation}
Then \eqref{eq5.4} holds.
As before, we shall later choose an \textit{optimal} value of $k$ for which \eqref{eq5.8} holds.

Again, we distinguish a few cases according to the special relations between various sets of intersection points and various special intervals.
We take advantage of the substantial overlapping of the short special intervals $J_k(q)$ defined by \eqref{eq3.5}.

Recall that if $1\le q',q''\le q_{k+1}-2$ are two integers such that $\{q'\alpha\}$ and $\{q''\alpha\}$ are neighboring points in the partition
$\AAA_k(\alpha)$, then the intersection $J_k(q')\cap J_k(q'')$ must contain a special interval of the type $J_{k+8}(q)$ for some $1\le q\le q_{k+9}-2$.
We split the argument into two complementary cases.

Write
\begin{displaymath}
Q(k;m)=\{y_i:q_{k+1}\le i\le m-q_{k+1}\}.
\end{displaymath}

\begin{case1a}
The following \textit{intersection property} holds.
For every
\begin{displaymath}
J_{k+8}(\ell)\subset J_k(q)
\quad\mbox{and}\quad
r\in\{0,1,\ldots,s-1\},
\end{displaymath}
with $r\ne r_1^{(0)}(q)$ given by \eqref{eq5.6}, we have
\begin{displaymath}
Q(k;m)\cap(r+J_{k+8}(\ell))\ne\emptyset.
\end{displaymath}
\end{case1a}

\begin{case1b}
There exist
\begin{displaymath}
J_{k+8}(\ell_2)\subset J_k(q^{(1)})
\quad\mbox{and}\quad
r_1\in\{0,1,\ldots,s-1\},
\end{displaymath}
with $r_1\ne r_1^{(0)}(q^{(1)})$ given by \eqref{eq5.6}, such that
\begin{equation}\label{eq5.9}
Q(k;m)\cap(r_1+J_{k+8}(\ell_2))=\emptyset.
\end{equation}
\end{case1b}

Since $J_{k+8}(\ell_2)\subset J_k(q^{(1)})$, it clearly follows from \eqref{eq5.7} that
\begin{equation}\label{eq5.10}
Q(k;m)\cap(r_1^{(0)}(q^{(1)})+J_{k+8}(\ell_2))=\emptyset.
\end{equation}
Since \eqref{eq5.9} and \eqref{eq5.10} are analogs of \eqref{eq5.4}, the $T$-power expansion argument shows that for every integer
$1-\ell_2\le h\le q_{k+9}-2-\ell_2$, we have
\begin{align}
&
Q^{(1)}(k;m)\cap T^h(r_1+J_{k+8}(\ell_2))=\emptyset,
\label{eq5.11}
\\
&
Q^{(1)}(k;m)\cap T^h(r_1^{(0)}(q^{(1)})+J_{k+8}(\ell_2))=\emptyset,
\label{eq5.12}
\end{align}
where, corresponding to $Q^*(k;m)$ in Section~\ref{sec4}, we write
\begin{displaymath}
Q^{(1)}(k;m)=\{y_i:q_{k+1}+q_{k+9}\le i\le m-q_{k+1}-q_{k+9}\}.
\end{displaymath}
For notational convenience, for every integer $1-\ell_2\le h\le q_{k+9}-2-\ell_2$, we write
\begin{align}
T^h(r_1^{(0)}(q^{(1)})+J_{k+8}(\ell_2))
&
=r_2^{(0)}(\ell_2+h)+J_{k+8}(\ell_2+h),
\label{eq5.13}
\\
T^h(r_1+J_{k+8}(\ell_2))
&
=r_2^{(1)}(\ell_2+h)+J_{k+8}(\ell_2+h).
\label{eq5.14}
\end{align}
Then combining \eqref{eq5.11}--\eqref{eq5.14}, we have
\begin{align}
&
Q^{(1)}(k;m)\cap(r_2^{(0)}(q)+J_{k+8}(q)=\emptyset,
\label{eq5.15}
\\
&
Q^{(1)}(k;m)\cap(r_2^{(1)}(q)+J_{k+8}(q)=\emptyset,
\label{eq5.16}
\end{align}
for every integer $1\le q\le q_{k+9}-2$.
Clearly
\begin{displaymath}
r_2^{(0)}(q)\ne r_2^{(1)}(q),
\quad
1\le q\le q_{k+9}-2.
\end{displaymath}
We now split Case~1B into two complementary cases.

\begin{case2a}
The following \textit{intersection property} holds.
For every
\begin{displaymath}
J_{k+16}(\ell)\subset J_{k+8}(q)
\quad\mbox{and}\quad
r\in\{0,1,\ldots,s-1\},
\end{displaymath}
with $r\ne r_2^{(0)}(q),r_2^{(1)}(q)$ given by \eqref{eq5.13} and \eqref{eq5.14}, we have
\begin{displaymath}
Q^{(1)}(k;m)\cap(r+J_{k+16}(\ell))\ne\emptyset.
\end{displaymath}
\end{case2a}

\begin{case2b}
There exist
\begin{displaymath}
J_{k+16}(\ell_3)\subset J_{k+8}(q^{(2)})
\quad\mbox{and}\quad
r_2\in\{0,1,\ldots,s-1\},
\end{displaymath}
with $r_2\ne r_2^{(0)}(q^{(2)}),r_2^{(1)}(q^{(2)})$ given by \eqref{eq5.13} and \eqref{eq5.14}, such that
\begin{equation}\label{eq5.17}
Q^{(1)}(k;m)\cap(r_2+J_{k+16}(\ell_3))=\emptyset.
\end{equation}
\end{case2b}

Since $J_{k+16}(\ell_3)\subset J_{k+8}(q^{(2)})$, it clearly follows from \eqref{eq5.15} and \eqref{eq5.16} that
\begin{align}
&
Q^{(1)}(k;m)\cap(r_2^{(0)}(q^{(2)})+J_{k+16}(\ell_3))=\emptyset,
\label{eq5.18}
\\
&
Q^{(1)}(k;m)\cap(r_2^{(1)}(q^{(2)})+J_{k+16}(\ell_3))=\emptyset.
\label{eq5.19}
\end{align}
Since \eqref{eq5.17}--\eqref{eq5.19} are analogs of \eqref{eq5.4}, the $T$-power expansion argument shows that for every integer
$1-\ell_3\le h\le q_{k+17}-2-\ell_3$, we have
\begin{align}
&
Q^{(2)}(k;m)\cap T^h(r_2+J_{k+16}(\ell_3))=\emptyset,
\label{eq5.20}
\\
&
Q^{(2)}(k;m)\cap T^h(r_2^{(0)}(q^{(2)})+J_{k+16}(\ell_3))=\emptyset,
\label{eq5.21}
\\
&
Q^{(2)}(k;m)\cap T^h(r_2^{(1)}(q^{(2)})+J_{k+16}(\ell_3))=\emptyset,
\label{eq5.22}
\end{align}
where, corresponding to $Q^{**}(k;m)$ in Section~\ref{sec4}, we write
\begin{displaymath}
Q^{(2)}(k;m)=\{y_i:q_{k+1}+q_{k+9}+q_{k+17}\le i\le m-q_{k+1}-q_{k+9}-q_{k+17}\}.
\end{displaymath}
For notational convenience, for every integer $1-\ell_3\le h\le q_{k+17}-2-\ell_3$, we write
\begin{align}
T^h(r_2^{(0)}(q^{(2)})+J_{k+16}(\ell_3))
&
=r_3^{(0)}(\ell_3+h)+J_{k+16}(\ell_3+h),
\label{eq5.23}
\\
T^h(r_2^{(1)}(q^{(2)})+J_{k+16}(\ell_3))
&
=r_3^{(1)}(\ell_3+h)+J_{k+16}(\ell_3+h),
\label{eq5.24}
\\
T^h(r_2+J_{k+16}(\ell_3))
&
=r_3^{(2)}(\ell_3+h)+J_{k+16}(\ell_3+h).
\label{eq5.25}
\end{align}
Then combining \eqref{eq5.20}--\eqref{eq5.25}, we have
\begin{align}
&
Q^{(2)}(k;m)\cap(r_3^{(0)}(q)+J_{k+16}(q)=\emptyset,
\nonumber
\\
&
Q^{(2)}(k;m)\cap(r_3^{(1)}(q)+J_{k+16}(q)=\emptyset,
\nonumber
\\
&
Q^{(2)}(k;m)\cap(r_3^{(2)}(q)+J_{k+16}(q)=\emptyset,
\nonumber
\end{align}
for every integer $1\le q\le q_{k+17}-2$.
Clearly
\begin{displaymath}
r_3^{(0)}(q),r_3^{(1)}(q),r_3^{(2)}(q)\mbox{ are distinct},
\quad
1\le q\le q_{k+17}-2.
\end{displaymath}
We now split Case~2B into two complementary cases.

\begin{case3a}
The following \textit{intersection property} holds.
For every
\begin{displaymath}
J_{k+24}(\ell)\subset J_{k+16}(q)
\quad\mbox{and}\quad
r\in\{0,1,\ldots,s-1\},
\end{displaymath}
with $r\ne r_3^{(0)}(q),r_3^{(1)}(q),r_3^{(2)}(q)$ given by \eqref{eq5.23}--\eqref{eq5.25}, we have
\begin{displaymath}
Q^{(2)}(k;m)\cap(r+J_{k+24}(\ell))\ne\emptyset.
\end{displaymath}
\end{case3a}

\begin{case3b}
There exist
\begin{displaymath}
J_{k+24}(\ell_4)\subset J_{k+8}(q^{(3)})
\quad\mbox{and}\quad
r_3\in\{0,1,\ldots,s-1\},
\end{displaymath}
with $r_3\ne r_3^{(0)}(q),r_3^{(1)}(q),r_3^{(2)}(q)$ given by \eqref{eq5.23}--\eqref{eq5.25}, such that
\begin{displaymath}
Q^{(2)}(k;m)\cap(r_3+J_{k+24}(\ell_4))=\emptyset.
\end{displaymath}
\end{case3b}

Suppose that $1\le\tau\le s-2$.
Assume that for every integer $1\le q\le q_{k+8\tau-7}-2$, there are distinct integers
\begin{equation}\label{eq5.26}
r_\tau^{(0)}(q),\ldots,r_\tau^{(\tau-1)}(q)\in\{0,1,\ldots,s-1\}
\end{equation}
such that for every integer $1\le j\le\tau$,
\begin{equation}\label{eq5.27}
Q^{(\tau-1)}(k;m)\cap(r_\tau^{(j-1)}(q)+J_{k+8\tau-8}(q)=\emptyset
\end{equation}
for every integer $1\le q\le q_{k+8\tau-7}-2$, where
\begin{displaymath}
Q^{(\tau-1)}(k;m)=\left\{y_i:\sum_{u=1}^\tau q_{k+8u-7}\le i\le m-\sum_{u=1}^\tau q_{k+8u-7}\right\}.
\end{displaymath}
Assume further that we have two complementary cases.

\begin{casetau0a}
The following \textit{intersection property} holds.
For every
\begin{displaymath}
J_{k+8\tau}(\ell)\subset J_{k+8\tau-8}(q)
\quad\mbox{and}\quad
r\in\{0,1,\ldots,s-1\},
\end{displaymath}
with $r\ne r_\tau^{(0)}(q),\ldots,r_\tau^{(\tau-1)}(q)$ given by \eqref{eq5.26}, we have
\begin{displaymath}
Q^{(\tau-1)}(k;m)\cap(r+J_{k+8\tau}(\ell))\ne\emptyset.
\end{displaymath}
\end{casetau0a}

\begin{casetau0b}
There exist
\begin{displaymath}
J_{k+8\tau}(\ell_{\tau+1})\subset J_{k+8\tau-8}(q^{(\tau)})
\quad\mbox{and}\quad
r_\tau\in\{0,1,\ldots,s-1\},
\end{displaymath}
with $r_\tau\ne r_\tau^{(0)}(q),\ldots,r_\tau^{(\tau-1)}(q)$ given by \eqref{eq5.26}, such that
\begin{equation}\label{eq5.28}
Q^{(\tau-1)}(k;m)\cap(r_\tau+J_{k+8\tau}(\ell_{\tau+1}))=\emptyset.
\end{equation}
\end{casetau0b}

Since $J_{k+8\tau}(\ell_{\tau+1})\subset J_{k+8\tau-8}(q^{(\tau)})$, it clearly follows from \eqref{eq5.27} that for every integer $1\le j\le\tau$,
\begin{equation}\label{eq5.29}
Q^{(\tau-1)}(k;m)\cap(r_\tau^{(j-1)}(q^{(\tau)})+J_{k+8\tau}(\ell_{\tau+1}))=\emptyset.
\end{equation}
Since \eqref{eq5.28} and \eqref{eq5.29} are analogs of \eqref{eq5.4}, the $T$-power expansion argument shows that for every integer
$1-\ell_{\tau+1}\le h\le q_{k+8\tau+1}-2-\ell_{\tau+1}$, we have
\begin{equation}\label{eq5.30}
Q^{(\tau)}(k;m)\cap T^h(r_\tau+J_{k+8\tau}(\ell_{\tau+1}))=\emptyset,
\end{equation}
as well as
\begin{equation}\label{eq5.31}
Q^{(\tau)}(k;m)\cap T^h(r_\tau^{(j-1)}(q^{(\tau)})+J_{k+8\tau}(\ell_{\tau+1}))=\emptyset
\end{equation}
for every integer $1\le j\le\tau$, where
\begin{displaymath}
Q^{(\tau)}(k;m)=\left\{y_i:\sum_{u=1}^{\tau+1}q_{k+8u-7}\le i\le m-\sum_{u=1}^{\tau+1}q_{k+8u-7}\right\}.
\end{displaymath}
For notational convenience, for every integer $1-\ell_{\tau+1}\le h\le q_{k+8\tau+1}-2-\ell_{\tau+1}$, we write
\begin{equation}\label{eq5.32}
T^h(r_\tau^{(j-1)}(q^{(\tau)})+J_{k+8\tau}(\ell_{\tau+1}))=r_{\tau+1}^{(j-1)}(\ell_{\tau+1}+h)+J_{k+8\tau}(\ell_{\tau+1}+h)
\end{equation}
for every integer $1\le j\le\tau$, and also write
\begin{equation}\label{eq5.33}
T^h(r_\tau+J_{k+8\tau}(\ell_{\tau+1}))=r_{\tau+1}^{(\tau)}(\ell_{\tau+1}+h)+J_{k+8\tau}(\ell_{\tau+1}+h).
\end{equation}
Then combining \eqref{eq5.30}--\eqref{eq5.33}, we have, for every integer $1\le j\le\tau+1$,
\begin{displaymath}
Q^{(\tau)}(k;m)\cap(r_{\tau+1}^{(j-1)}(q)+J_{k+8\tau}(q)=\emptyset
\end{displaymath}
for every integer $1\le q\le q_{k+8\tau+1}-2$.
Clearly
\begin{displaymath}
r_{\tau+1}^{(0)}(q),\ldots,r_{\tau+1}^{(\tau)}(q)\mbox{ are distinct},
\quad
1\le q\le q_{k+8\tau+1}-2.
\end{displaymath}
We now split Case~$\tau$B into two complementary cases.

\begin{casetau1a}
The following \textit{intersection property} holds.
For every
\begin{displaymath}
J_{k+8\tau+8}(\ell)\subset J_{k+8\tau}(q)
\quad\mbox{and}\quad
r\in\{0,1,\ldots,s-1\},
\end{displaymath}
with $r\ne r_{\tau+1}^{(0)}(q),\ldots,r_{\tau+1}^{(\tau)}(q)$ given by \eqref{eq5.32} and \eqref{eq5.33}, we have
\begin{displaymath}
Q^{(\tau)}(k;m)\cap(r+J_{k+8\tau+8}(\ell))\ne\emptyset.
\end{displaymath}
\end{casetau1a}

\begin{casetau1b}
There exist
\begin{displaymath}
J_{k+8\tau+8}(\ell_{\tau+2})\subset J_{k+8\tau}(q^{(\tau+1)})
\quad\mbox{and}\quad
r_{\tau+1}\in\{0,1,\ldots,s-1\},
\end{displaymath}
with $r_{\tau+1}\ne r_{\tau+1}^{(0)}(q),\ldots,r_{\tau+1}^{(\tau)}(q)$ given by \eqref{eq5.32} and \eqref{eq5.33}, such that
\begin{displaymath}
Q^{(\tau)}(k;m)\cap(r_{\tau+1}+J_{k+8\tau+8}(\ell_{\tau+2}))=\emptyset.
\end{displaymath}
\end{casetau1b}

In particular, if $\tau=s-2$, we have the following final case.

\begin{cases1b}
There exist
\begin{displaymath}
J_{k+8s-8}(\ell_s)\subset J_{k+8s-16}(q^{(s-1)})
\quad\mbox{and}\quad
r_{s-1}\in\{0,1,\ldots,s-1\},
\end{displaymath}
with $r_{s-1}\ne r_{s-1}^{(0)}(q^{(s-1)}),\ldots,r_{s-1}^{(s-2)}(q^{(s-1)})$ given by \eqref{eq5.32} and \eqref{eq5.33} in the special case $\tau=s-2$, such that
\begin{equation}\label{eq5.34}
Q^{(s-2)}(k;m)\cap(r_{s-1}+J_{k+8s-8}(\ell_s))=\emptyset.
\end{equation}
\end{cases1b}

\begin{lem}\label{lem51}
For every $\tau=1,\ldots,s-1$, Case $\tau$A is impossible.
\end{lem}

\begin{lem}\label{lem52}
If Case $(s-1)$B holds, then
\begin{equation}\label{eq5.35}
m\le2s+2\sum_{u=1}^sq_{k+8u-7}.
\end{equation}
\end{lem}

Before we prove Lemmas \ref{lem51} and~\ref{lem52}, we first complete the proof of Theorem~\ref{thm1}.

\begin{proof}[Proof of Theorem~\ref{thm1}]
Suppose that $I_0\subset[0,s)$ is $(\LLL_\alpha;M)$-free, so that
\begin{displaymath}
\{y_i:1\le i\le m\}\cap I_0=\emptyset.
\end{displaymath}
Let $J_k(\ell_1)=J(\alpha;k;\ell_1)$ be the \textit{longest} special interval of the form \eqref{eq3.5} such that
\begin{displaymath}
r(\ell_1)+J_k(\ell_1)\subset I_0
\quad\mbox{for some $\ell_1$ and $r(\ell_1)\in\{0,1,\ldots,s-1\}$}.
\end{displaymath}
Then we can show as before that
\begin{equation}\label{eq5.36}
\length(I_0)<\frac{8}{q_{k}}.
\end{equation}

On the other hand, it follows from \eqref{eq5.1} and \eqref{eq5.2} that
\begin{equation}\label{eq5.37}
M\le(m+1)\sqrt{1+\alpha^2}.
\end{equation}
Also, in view of Lemmas \ref{lem51} and~\ref{lem52}, it is clear that the bound \eqref{eq5.35} holds.
Finally, the inequalities \eqref{eq3.33} are replaced by the inequalities
\begin{equation}\label{eq5.38}
q_{k+1}<q_{k+9}<\ldots<q_{k+8s-7}\le(A+1)^{8s-7}q_k.
\end{equation}
Combining \eqref{eq5.35}, \eqref{eq5.37} and \eqref{eq5.38}, we see that
\begin{equation}\label{eq5.39}
M\le\left(2s+1+2\sum_{u=1}^sq_{k+8u-7}\right)\sqrt{1+\alpha^2}
<(4s+1)(A+1)^{8s-7}q_k\sqrt{2}.
\end{equation}

It now follows from \eqref{eq5.36} and \eqref{eq5.39} that a geodesic segment $\LLL_{\alpha}(t)$, $0\le t\le M$, of length
$M=(4s+1)(A+1)^{8s-7}q_k\sqrt{2}$ must intersect every subinterval $I$ of $h_1\cup\ldots\cup h_s$ with $\length(I)=8/q_{k}$.
Since the product $M\length(I)$ is a constant independent of~$k$, this establishes superdensity of the half-infinite geodesic.
\end{proof}

It remains to prove Lemmas \ref{lem51} and~\ref{lem52}.

\begin{proof}[Proof of Lemma~\ref{lem51}]
We can proceed along similar lines as in the first part of the proof of Lemma~\ref{lem31}, and show that any two neighboring $s$-copy extensions $J_{k+8\tau-8}(q';s)$ and $J_{k+8\tau-8}(q'';s)$ are \textit{synchronized} in the following precise sense:
For each $r\in\{0,1,\ldots,s-1\}$, we have
\begin{displaymath}
Q^{(\tau-1)}(k;m)\cap(r+J_{k+8\tau-8}(q'))=\emptyset
\end{displaymath}
if and only if
\begin{displaymath}
Q^{(\tau-1)}(k;m)\cap(r+J_{k+8\tau-8}(q''))=\emptyset.
\end{displaymath}

By definition, the $s$-copy extensions $J_{k+8\tau-8}(q;s)$, $1\le q\le q_{k+8\tau-7}-2$, give rise to $2s$ continuous chains of overlapping intervals in the torus/circle $[0,s)$ such that the $2s$ chains completely cover the $2s$ intervals
\begin{equation}\label{eq5.40}
[0,1-\alpha),
\quad
[1-\alpha,1),
\quad
\ldots,
\quad
[s-1,s-\alpha),
\quad
[s-\alpha,s).
\end{equation}
The synchronization property now implies that each of the $2s$ long special intervals in \eqref{eq5.40} satisfies one of the following two properties.
Either such a long special interval is \textit{disjoint} from the set $Q^{(\tau-1)}(k;m)$, or the set $Q^{(\tau-1)}(k;m)$ is \textit{dense} in such a long special interval, in the precise sense that every subinterval of length $1/q_{k+8s}$ contains a point from the set $Q^{(\tau-1)}(k;m)$.

Moreover, it is not difficult to show that precisely $2\tau$ of the $2s$ long special intervals in \eqref{eq5.40} are disjoint from the set
$Q^{(\tau-1)}(k;m)$.
To see this, choose two integers $q'$ and $q''$ satisfying $1\le q',q''\le q_{k+8\tau-7}-2$ such that
\begin{equation}\label{eq5.41}
J_{k+8\tau-8}(q')\subset(0,1-\alpha)
\quad\mbox{and}\quad
J_{k+8\tau-8}(q'')\subset(1-\alpha,1).
\end{equation}
Then it follows from \eqref{eq5.27} that for every integer $1\le j\le\tau$,
\begin{align}
Q^{(\tau-1)}(k;m)\cap(r_\tau^{(j-1)}(q')+J_{k+8\tau-8}(q'))
&=\emptyset,
\label{eq5.42}
\\
Q^{(\tau-1)}(k;m)\cap(r_\tau^{(j-1)}(q'')+J_{k+8\tau-8}(q''))
&=\emptyset.
\label{eq5.43}
\end{align}
For every integer $1\le j\le\tau$, now write
\begin{equation}\label{eq5.44}
\III_{2j-1}=r_\tau^{(j-1)}(q')+(0,1-\alpha)
\quad\mbox{and}\quad
\III_{2j}=r_\tau^{(j-1)}(q'')+(1-\alpha,1).
\end{equation}
The synchronization property and \eqref{eq5.41}--\eqref{eq5.44} now imply that for every integer $1\le j\le\tau$, we have
\begin{displaymath}
Q^{(\tau-1)}(k;m)\cap\III_{2j-1}=\emptyset
\quad\mbox{and}\quad
Q^{(\tau-1)}(k;m)\cap\III_{2j}=\emptyset.
\end{displaymath}
Note that the union $\III_1\cup\ldots\cup\III_{2\tau}$ modulo~$1$ is precisely the unit interval $[0,1)$ taken $\tau$ times.

Now $\III_1,\ldots,\III_{2\tau}$ are $2\tau$ of the $2s$ long special intervals in \eqref{eq5.40}.
Let the remaining long special intervals in \eqref{eq5.40} be denoted by $\III_{2j-1}$ and $\III_{2j}$, $\tau<j\le s$.
The condition of Case~$\tau$A now implies that these $2s-2\tau$ intervals are not disjoint from $Q^{(\tau-1)}(k;m)$, so that $Q^{(\tau-1)}(k;m)$ is dense in each of them.

Each $T$-image $T(\III_{2j-1})$ and $T(\III_{2j})$, $1\le j\le\tau$, has at most $1$ common point with the set $Q^{(\tau-1)}(k;m)$.
This is a contradiction, since the union
\begin{displaymath}
T(\III_1)\cup\ldots)\cup T(\III_{2\tau})
\end{displaymath}
has a \textit{substantial} intersection with the union $\III_{2\tau+1}\cup\ldots\cup\III_{2s}$, which implies that it must have a \textit{substantial} intersection with the set $Q^{(\tau-1)}(k;m)$, much more than at most $2\tau$ elements.
Thus Case~$\tau$A is impossible, and this completes the proof.
\end{proof}

\begin{proof}[Proof of Lemma~\ref{lem52}]
Since \eqref{eq5.34} is an analog of \eqref{eq5.4}, we can repeat the $T$-power extension argument and conclude that for every integer
$1-\ell_s\le h\le q_{k+8s-7}-2-\ell_s$, we have
\begin{equation}\label{eq5.45}
Q^{(s-1)}(k;m)\cap T^h(r_{s-1}+J_{k+8s-8}(\ell_s))=\emptyset.
\end{equation}
Since $J_{k+8s-8}(\ell_s)\subset J_{k+8s-16}(q^{(s-1)})$, it follows from \eqref{eq5.27} with $\tau=s-1$ that for every integer $1\le j\le s-1$,
\begin{equation}\label{eq5.46}
Q^{(s-2)}(k;m)\cap(r_{s-1}^{(j-1)}(q^{(s-1)})+J_{k+8s-8}(\ell_s))=\emptyset.
\end{equation}
Next, note that \eqref{eq5.46} are also analogs of \eqref{eq5.4}, so again we can repeat the $T$-power extension argument and conclude that for every integer $1\le j\le s-1$, and for every integer $1-\ell_s\le h\le q_{k+8s-7}-2-\ell_s$, we have
\begin{equation}\label{eq5.47}
Q^{(s-1)}(k;m)\cap T^h(r_{s-1}^{(j-1)}(q^{(s-1)})+J_{k+8s-8}(\ell_s))=\emptyset.
\end{equation}

Now, for every integer $1-\ell_s\le h\le q_{k+8s-7}-2-\ell_s$, we write
\begin{equation}\label{eq5.48}
T^h(r_{s-1}^{(j-1)}(q^{(s-1)})+J_{k+8s-8}(\ell_s))=r^{(j-1)}(\ell_s+h)+J_{k+8s-8}(\ell_s+h)
\end{equation}
for every integer $1\le j\le s-1$, and also write
\begin{equation}\label{eq5.49}
T^h(r_{s-1}+J_{k+8s-8}(\ell_s))=r^{(s-1)}(\ell_s+h)+J_{k+8s-8}(\ell_s+h).
\end{equation}
Clearly it follows from the assumption of Case~$(s-1)$B that $r^{(0)}(q),\ldots,r^{(s-1)}(q)$ form a permutation of $0,1,\ldots,s-1$ for every integer
$1\le q\le q_{k+8s-7}-2$.
Combining \eqref{eq5.45} and \eqref{eq5.47}--\eqref{eq5.49}, we have, for every integer $1\le j\le s$,
\begin{equation}\label{eq5.50}
Q^{(P-1)}(k;m)\cap(r^{(j-1)}(q)+J_{k+8s-8}(q))=\emptyset
\end{equation}
for every integer $1\le q\le q_{k+8s-7}-2$.

Note now that \eqref{eq5.50} are similar to \eqref{eq4.7}--\eqref{eq4.9} in the proof of Lemma~\ref{lem41}, so we now mimic the last part of that proof.

By definition, the $s$-copy extensions $J_{k+8s-8}(q;s)$, $1\le q\le q_{k+8s-7}-2$, give rise to $2s$ continuous chains of overlapping intervals in the
torus/circle $[0,s)$ such that the $2s$ chains completely cover the $2s$ intervals
\begin{displaymath}
[0,1-\alpha),
\quad
[1-\alpha,1),
\quad
\ldots,
\quad
[s-1,s-\alpha),
\quad
[s-\alpha,s),
\end{displaymath}
and there are only $2s$ points in $[0,s)$ that are not covered by the $2s$ chains, namely
\begin{displaymath}
j-1,
\quad
j-\alpha,
\quad
1\le j\le s.
\end{displaymath}

Combining \eqref{eq5.50} for every integer $1\le q\le q_{k+8s-7}-2$, we deduce that the set $Q^{(s-1)}(k;m)$ is not covered by the $2s$ chains. 
Indeed, if
\begin{displaymath}
m\ge2s+1+2\sum_{u=1}^sq_{k+8u-7},
\end{displaymath}
then the set $Q^{(s-1)}(k;m)$ has at least $2s+1$ distinct elements, which is more than~$2s$, giving rise to a contradiction. 
We conclude therefore that, under the conditions of Case~$(s-1)$B, we must have
\begin{displaymath}
m\le2s+2\sum_{u=1}^sq_{k+8u-7},
\end{displaymath}
and this completes the proof.
\end{proof}

%
%


\begin{thebibliography}{99}

\bibitem{BDY1}
J. Beck, M. Donders, Y. Yang.
Quantitative behavior of non-integrable systems (I).
\textit{Acta Math. Hungar.} \textbf{161} (2020), 66--184.

\bibitem{BDY2}
J. Beck, M. Donders, Y. Yang.
Quantitative behavior of non-integrable systems (II).
\textit{Acta Math. Hungar.} \textbf{162} (2020), 220--324.

\bibitem{BCY1}
J. Beck, W.W.L. Chen, Y. Yang.
Quantitative behavior of non-integrable systems (III).
\\
\texttt{arxiv.org/abs/2006.06213}, 93 pp.

\bibitem{BCY2}
J. Beck, W.W.L. Chen, Y. Yang.
Quantitative behavior of non-integrable systems (IV).
\\
\texttt{arxiv.org/abs/2012.12038}, 120 pp.

\bibitem{H}
J.H. Halton.
The distribution of the sequence $\{n\xi\}$ ($n=0,1,2,\ldots$).
\textit{Proc. Cambridge Philos. Soc.} \textbf{61} (1965), 665--670.

\bibitem{KZ75}
A. Katok, A. Zemlyakov.
Topological transitivity of billiards in polygons.
\textit{Math. Notes} \textbf{18} (1975), 760--764.

\bibitem{K2}
A.Ya. Khinchin.
\textit{Continued Fractions} (Dover, 1997).

\bibitem{KS}
D. K\"{o}nig, A. Sz\"{u}cs.
Mouvement d'un point abondonne a l'interieur d'un cube.
\textit{Rend. Circ. Mat. Palermo} \textbf{36} (1913), 79--90.

\bibitem{Sl}
N.B. Slater.
Gap and steps for the sequence $n\theta$ mod~$1$.
\textit{Proc. Cambridge Philos. Soc.} \textbf{63} (1967), 1115-1123.

\bibitem{So1}
V.T. S\'{o}s.
On the theory of diophantine approximations.
\textit{Acta Math. Acad. Sci. Hungar.} \textbf{8} (1957), 461--472.

\bibitem{So2}
V.T. S\'{o}s.
On the distribution mod~$1$ of the sequence $n\alpha$.
\textit{Ann. Univ. Sci. Budapest E\"{o}tv\"{o}s Sect. Math.} \textbf{1} (1958), 127--134.

\bibitem{Su}
J. Sur\'{a}nyi.
\"{U}ber die Anordnung der Vielfachen einer reellen Zahl mod~$1$.
\textit{Ann. Univ. Sci. Budapest E\"{o}tv\"{o}s Sect. Math.} \textbf{1} (1958), 107--111.

\bibitem{Sw}
S. Swierczkowski.
On successive settings of an arc on the circumference of a circle.
\textit{Fund. Math.} \textbf{46} (1959), 187--189.

\end{thebibliography}
\end{document}